\documentclass[11pt,final]{amsart}

\usepackage{amsmath}
\usepackage{amsfonts}
\usepackage{amssymb}
\usepackage{amscd}
\usepackage{amsthm}
\usepackage{mathrsfs}
\usepackage{framed}
\usepackage{fullpage}
\usepackage{graphicx}
\usepackage{latexsym}
\usepackage{tikz-cd}
\usepackage[numbers]{natbib}
\usepackage{multirow}
\usepackage{tikz}
\usepackage{xcolor}
\usepackage{mathtools}
\usepackage{fancybox}
\usepackage[nomargin,inline]{fixme}
\usepackage{braket}
\usetikzlibrary{positioning}
\usetikzlibrary{arrows}

\usepackage{xcolor}
\definecolor{Pantone291}{RGB}{105,179,231} 

\usepackage{hyperref}
\hypersetup{colorlinks,linkcolor={red},citecolor={blue},urlcolor={blue}}
\usepackage{cleveref}
\crefrangeformat{equation}{(#3#1#4--#5#2#6)}
\usepackage{showkeys}

\numberwithin{equation}{section}

\newtheorem{Theorem}{Theorem}[section]
\newtheorem*{Theorem*}{Theorem}
\newtheorem{Corollary}[Theorem]{Corollary}
\newtheorem{Lemma}[Theorem]{Lemma}
\newtheorem{Proposition}[Theorem]{Proposition}
\newtheorem{Definition}[Theorem]{Definition}

\newtheorem{Remark}[Theorem]{Remark}

\numberwithin{equation}{section}

\newcommand{\CC}{\mathbb C}

\newcommand{\RR}{\mathbb R}
\newcommand{\ZZ}{\mathbb Z}

\newcommand{\cA}{\mathcal A}

\newcommand{\cT}{\mathcal T}

\newcommand{\cR}{\mathcal R}

\newcommand{\cP}{\mathcal P}
\newcommand{\cM}{\mathcal M}

\newcommand{\bN}{\mathbf N}

\newcommand{\Weyl}{W}

\newcommand{\la}{\lambda}
\newcommand{\si}{\sigma}

\newcommand{\bs}{\mathbf{s}}

\newcommand{\ms}{{\underline{\mathbf{m}}}}
\newcommand {\Fl} {{F^\vee}}

\newcommand{\ft}{\mathfrak{t}}
\newcommand{\ff}{\mathfrak{f}}
\newcommand{\ab}{\mathbf{ab}}

\newcommand{\GL}{\mathop{\mathrm {GL}}\nolimits}

\def\Tr{\operatorname{Tr}}
\def\End{\operatorname{End}}

\newcommand{\cO}{\mathcal O}
\newcommand{\cK}{\mathcal K}
\newcommand{\weights}{\operatorname{wt}}
\newcommand{\something}{\aleph}
\newcommand{\GP}{\mathcal{G} \mathcal{P}}
\newcommand{\bV}{ \mathbf{V}}
\newcommand{\A}{ \mathbf{A}_d}
\newcommand{\pt}{\mathrm{pt}}

\newcommand{\IGN}{|v' \rangle}
\newcommand{\bign}{\langle v|}

\newcommand{\Sinv}{_{\text{loc}}}

\newcommand{\CH}{H}

\newcommand{\C}{C}
\newcommand{\D}{\Delta}

\newcommand{\Gr}{\operatorname{Gr}}
\newcommand{\FP}{E \cP( \RR \times \ZZ)}

\newcommand{\m}{\underline{\mathbf{m}}}

\makeatletter
\renewcommand*\FXLayoutInline[3]{%
  {\@fxuseface{inline}
  \ignorespaces\noindent \ovalbox{\hspace{.01\textwidth} \begin{minipage}{.95\textwidth}
  	#3 \fxnotename{#1}: #2
  \end{minipage}\hspace{.01\textwidth}}}
  \newline}
\makeatother

\title{Analytical Traces on Coulomb Branches of Quiver Gauge Theories }
\author{Keke Zhang}
\date{}

\begin{document}

\maketitle
\begin{abstract}
   In this paper, we present an explicit construction of twisted traces for quantum Coulomb branches of conical theories. We develop an operator representation of the Coulomb branch algebra and use it to derive integral formulas for the twisted trace. Our construction provides a concrete realization of twisted traces that arise as the correlation functions of a conformal field theory, particularly in the work of Beem, Peelaers, and Rastelli. This complements recent developments in the study of twisted traces on quantum Higgs branches and offers new mathematical insights into the structure of quantum Coulomb branches.
\end{abstract}

\section{Introduction}

\subsection{Background}

Three-dimensional \(\mathcal{N}=4\) supersymmetric gauge theories, characterized by a complex reductive group \(G\) and representation \(\mathbf{N}\), exhibit rich mathematical structures through their Higgs and Coulomb branches. A cornerstone of this framework is 3D mirror symmetry, which predicts dualities exchanging these branches. Considerations from supersymmetric field theory predict that when the 3d theory has a conformal limit, both the Higgs and Coulomb branches carry a twisted trace which encodes the correlation functions of that conformal field theory. While twisted traces on quantized algebras are fundamental to this symmetry—encoding correlation functions and enabling the construction of short star-products—explicit realizations have been largely confined to the Higgs branch.  See \cref{context} for the mathematical definitions.

Recent advances by Gaiotto et al. \cite{BenWinter} established twisted traces for quantum Higgs branches. However, it is unknown how
 the discussions in the physics literature interface at all with the BFN construction of the Coulomb branches,
despite their critical role in mirror symmetry and the assumptions of their existence in physics works like Beem–Peelaers–Rastelli \cite{Beem_2017} and their mathematical importance for the quantum Hikita conjecture \cite{kamnitzer2020quantum}.  

This paper closes that gap. 
We construct an explicit operator representation of the quantized Coulomb branch algebra \(\mathcal{A}_h(G,\mathbf{N})\) (\Cref{333}) for conical theories (\Cref{sec:operrep}).
Using this representation, we derive integral formulas for twisted traces, providing a concrete realization of objects previously assumed in physical arguments (\Cref{sec:twtr}).

\subsection{Statement of Theorem}
In this paper, the central object is the quantized Coulomb branch algebra, \( \mathcal{A}_\hbar(G, \mathbf{N}) \), as defined in \cite[1(iii)]{braverman2019mathematical}.  See \Cref{sec2.2} for more details. 
\begin{Definition} \label{1.1}
    The pair \((G, \bN)\) is \textbf{conical} if for all 
    \(\beta \in \mathrm{t}\), 
    \begin{equation}
   \sum_{j=1}^n |a_j(\beta)| > \sum_{\alpha \in \Delta}|\alpha(\beta)|\end{equation}
    (per \cite[(A.1)]{muthiah2022}).

    Following \cite[Lemma 2.6]{BenWinter}, the conical condition is equivalent to the positivity of the \(\Delta\)-grading on \(\mathcal{A}_{\hbar}\), as discussed in \cite[Remark 2.8(2)]{braverman2018coulomb}.
\end{Definition}

To consider deformations of this algebra, we need to introduce a larger group $\tilde{G}$, which contains $G$ as a normal subgroup such that the quotient group $F=\tilde{G}/G$ is a torus. 

Let \(\mathfrak{t}_{\tilde{G}} = \mathfrak{t}_G \oplus \ff\) denote a Cartan subalgebra of $\tilde{G}$ where \(\mathfrak{t}_G\) (resp. \(\ff\)) is a Cartan subalgebra of \(G\) (resp. \(F\)) (See \Cref{coord} for the choice of embedding).
The algebra $ \mathcal{A}_\hbar(G, \mathbf{N}) $ then implicitly depends upon a mass parameter denoted by \( \m \in \mathfrak{f}.\)

 This algebra is naturally endowed with an automorphism arising from the Fayet-Iliopoulos (FI) parameter of the theory; see Definition \ref{def} for the definition. Twisted trace depends on this automorphism.
The twisted trace of $a\in \mathcal{A}_\hbar(G, \mathbf{N})$ only depends on the projection of $a$ to the fixed points of the automorphism.
Based on grading considerations, we find that the trace factors through projection to a subalgebra isomorphic to \(\operatorname{Sym}((\mathfrak{t} +\m)^*)^W\), the algebra of $W$-invariant polynomials on the coset $\ft +\m$. 

Thus we need only describe the functional on this polynomial ring induced by the twisted trace. We will describe this below as the integral with respect to a measure on $\mathfrak{t}$, 
and the FI parameter \( \zeta \) , which is dual to the mass parameter \( \m \).
Consider the weights $w\in \weights(\mathbf{N})$ of $\mathbf{N}$ and roots $\beta\in \Delta$ of $G$ as linear functions on $\mathfrak{t}_{\tilde{G}} $. 
In order to define the measure, we use the meromorphic function   \begin{equation}  
    \mathrm{w}(\sigma + \m) = \frac{  
      \prod_{\beta \in \Delta^+_i} (\frac{1}{\pi} \sinh(\pi \beta)  )^2
    }{  
      \prod_{w \in \operatorname{wt}(\mathbf{N})} 
      \frac{1}{\pi} \cosh\left(\pi w(\sigma + \m)\right)  
    } \qquad \text{for } \sigma\in \ft_{\RR}. 
    \end{equation}

\begin{Definition}[\mbox{\cite[Section 3.1]{2020}}] \label{def:tr}
Let $A$ be an associative algebra, and $g: A \to A$ be an invertible linear map. A \textit{$g$ twisted trace} on $A$ is a linear functional $\Tr: A \to \mathbb{C}$ satisfying:
\begin{equation} \label{twtr}
  \Tr(ab) = \Tr(bg(a))  
\end{equation}
for all $a, b \in A$. 
If $\Tr$ is a twisted trace, then applying \eqref{twtr}, with $ b=1$ shows that 
\begin{equation} \Tr(g(a)) = \Tr(a).
\end{equation}  In particular, $\Tr$ must vanish on the image of the map $1-g$.
\end{Definition}
This generalizes the notion of a trace on the algebra $A$, which is a linear map satisfying Definition \ref{def:tr} when $g$ is the identity.

\begin{Theorem} \label{prop:vector}  
For a conical pair \((G, \mathbf{N})\) coming from a quiver gauge theory of type \( A \) with \( G = \prod_{i \in Q_0} \operatorname{GL}(V_i) \) (\( V_i = \mathbb{C}^{a_i} \)) (see \cref{def} for definition), the quantized Coulomb branch algebra \( \mathcal{A}_\hbar(G, \mathbf{N}) \) for a choice of $\m, \zeta$ admits a twisted trace $\Tr$ for the automorphism $\Omega$ defined on \( R(\sigma)\in \operatorname{Sym}((\mathfrak{t} +\m)^*)^W \) by the integral:
\begin{equation}  
\Tr(R(\sigma)) = \int_{\mathfrak{t}_G} e^{2\pi \zeta(\sigma)} R(\sigma) \, \mathrm{w}(\sigma + \m) \, d\sigma,  
\end{equation}  
\end{Theorem}

\subsection{}

We develop a concrete operator representation of the algebra $\mathcal{A}_h(\CC^*, \CC)$ on a specific space of functions $\mathcal{P}(\mathbb{R} \times \mathbb{Z})$ (and its multi-dimensional and Weyl-invariant generalizations).  This is both useful as an illustrative example, and one of the building blocks of the more general construction. Such an operator representation is achieved by defining shift operators ($X, P, \widetilde{X}, \widetilde{P}$) that act by shifting a continuous variable $\sigma$ (related to the Cartan subalgebra) and a discrete variable $b$, combined with multiplication by polynomial factors. This approach is related to the difference operator representations, discussed in \cite[Appendix A]{braverman2019mathematical}.  We generalize this to all abelian cases.

The construction for a general non-abelian group $G$ is built from the abelian case (where $G$ is a torus) by means of localization to torus fixed points. As a consequence, the Coulomb branch algebra for $G$ becomes isomorphic to the $W$-invariant part of the algebra for its maximal torus $T$ after inverting certain parameters. This allows us to define operators for non-abelian monopoles $\mathbf{V}_\lambda$ as symmetrized sums of the abelian shift operators.

A special ``ground state'' vector $|1\rangle_{G,\mathbf{N}}$ is defined within the representation space motivated by the physical theory \cite[(5.1)]{GaiottoTempFile}. The \textbf{twisted trace} of an element $w\in  \mathcal{A}_\hbar(G, \mathbf{N})$ is then defined via a pairing (an inner product on the function space) involving this ground state and the action of the operator $\alpha_C(w)$:
    $$\operatorname{Tr}(w) = \langle 1| \alpha_C(w) g |1\rangle_{G,\mathbf{N}}$$
    where $g$ is related to the automorphism and $\alpha_C$ is the operator representation.

By evaluating this abstract definition, the author proves that the twisted trace on polynomial elements $R(\sigma)$ in the Cartan subalgebra is given by an explicit integral formula (Theorem 1.3 and 4.1):
    $$
    \operatorname{Tr}(R(\sigma)) = \int e^{2\pi\zeta(\sigma)} R(\sigma)  \operatorname{w}(\sigma + m)  d\sigma
    $$
    The measure $\operatorname{w}(\sigma + m)$ is constructed from hyperbolic functions ($\sinh$, $\cosh$) determined by the roots of $G$ and the weights of the representation $\mathbf{N}$. The conical condition ensures this integral converges.

\subsection{Future Directions}

In the future work, we will extend these results from the cohomological setting (discussed here) to the K-theoretic setting. 
The K-theoretic Coulomb branch can be interpreted as capturing the algebra of BPS loop operators in the $4 \mathrm{~d} \mathcal{N}=2$ gauge theory, and the quantized version of it corresponds to deformation quantization of these operator algebras\cite{CautisWilliams2023}. 
The analogous result would be an operator representation of K-theoretic Coulomb branches acting on a space of functions constructed from $q$-Gamma functions. 
This naturally suggests  an approach to constructing twisted traces on K-theoretic Coulomb branches.

The twisted trace is closely related to the quantum Hikita conjecture. The quantum Hikita conjecture \cite[Conjecture 1.1]{kamnitzer2020quantum} states that for conical theories (see \Cref{1.1}), the specialized quantum D-modules on Higgs branches and the D-modules of graded traces for Coulomb branches are isomorphic as D-modules. 
For the graded trace, the grading can be thought of as arising from a torus action, where different torus elements give different twisted traces. 
We may use this twisted trace to prove a weaker version of the quantum Hikita conjecture.

\section{Context} \label{context}

\subsection{Nakajima Quiver and Higgs Branches} 
\begin{Definition} \cite[Section 2.1]{nakajima2016introquivervarieties}
A \textbf{framed quiver} is a triple $(Q, \mathbf{v}, \mathbf{w})$, where:
\begin{enumerate}
\item $Q = (I, E)$ is a directed graph (quiver) with a set of nodes $I$ and a set of edges (arrows) $E$.
\item $\mathbf{v} = (v_i)_{i \in I} \in \mathbb{Z}_{>0}^I$ is a \textit{gauge vector} that assigns a positive integer $v_i$ to each node $i \in I$, representing the dimension of a vector space $V_i$.
\item $\mathbf{w} = (w_i)_{i \in I} \in \mathbb{Z}_{\geq 0}^I$ is a \textit{framing vector} that assigns a non-negative integer $w_i$ to each node $i \in I$, representing the dimension of a framing vector space $W_i$.
\end{enumerate}
A stability parameter $\theta \in \mathbb{R}^I$ provides a stability condition that selects stable representations in the geometric invariant theory (GIT) quotient.
\end{Definition}
Most readers will be familiar with quiver gauge theories because of their connection to Nakajima quiver varieties.
\begin{Definition} \cite[Section 2.2]{nakajima2016introquivervarieties}\label{def}
Let $(Q, \mathbf{v}, \mathbf{w})$ be a framed quiver. The \textbf{Nakajima quiver variety} $\mathcal{M}(\mathbf{v}, \mathbf{w}, \theta)$ is constructed as follows:

\begin{enumerate}
\item For each node $i \in I$, we associate vector spaces $V_i$ of dimension $v_i$ and $W_i$ of dimension $w_i$, with maps $B_i: W_i \rightarrow V_i$ and $C_i: V_i \rightarrow W_i$.

\item Let $\boldsymbol{\mu}$ be the moment map derived from the quiver's representation theory.

\item The Nakajima quiver variety is defined as the GIT quotient:
$$\mathcal{M}(\mathbf{v}, \mathbf{w}, \theta) = \mu^{-1}(\zeta)/\!\!/\!_\theta G$$
where $G = \prod_{i \in I} GL(V_i)$ is the gauge group that acts on the space of quiver representation $\bN$ and $\theta \in \mathbb{R}^I$ is the stability parameter.
\end{enumerate}
Here, the parameter $\zeta$ is called the \textbf{FI parameter}.
The space $\mathcal{M}(\mathbf{v}, \mathbf{w}, \theta)$ also denoted as $\cM_H (G, \bN)$ is the \textbf{Higgs branch} of the gauge theory associated to $(G,\bN)$.
\end{Definition}

\subsection{Coulomb branch} \label{sec2.2} Instead of the Higgs branch of these theories, we will study the Coulomb branch.
We denote the formal power series as $\CC [[t]] $ and the formal Laurent series as $\CC((t)) $. For a reductive group $G$, we denote  $G$ points over these base rings by $G_\cO= G(\CC [[ t ]] ) $ and $G_\cK= G(\CC ((t)) ) $, and we denote the affine Grassmannian as $\Gr = G_\cO/ G_\cK $. We denote their representations as $\bN_\cO$ and $\bN_\cK$, respectively.  We follow the notation of \cite[\S 2(i)]{braverman2019mathematical} and let 
\[\cR = \left(G_\cK \times_ {G_\cO }\bN_\cO \right) \times_{\bN_\cK} \bN_\cO,\qquad\qquad  \cT=G_\cK \times_ {G_\cO }\bN_\cO.\]
We consider the stack
\begin{equation}
   \frac{\bN_\cO}{G_\cO} \times_{\frac{\bN_\cK}{G_\cK}} \frac{\bN_\cO}{G_\cO} 
  =  G_\cO \backslash \mathcal{R}.
\end{equation}
\begin{Definition} \label{333}
\cite[1(iii)]{braverman2019mathematical}
The \textbf{quantized Coulomb branch algebra} $\mathcal{A}_{\hbar} (G, \bN)$ is defined as the equivariant homology group $H_*^{G_{\mathcal{O}} \rtimes \mathbb{C}^{\times}}(\mathcal{R})$, where $\mathbb{C}^{\times}$ is the loop rotation group. There exists an algebra structure on it given by \cite[A(i)]{braverman2018coulomb}.
\end{Definition}
We identify $H^*_{\CC^\times}(pt) $ with $\CC[\hbar]$ as seen in \cite[1(iii)]{braverman2019mathematical}.
See \cite[Example 3.9--3.11]{webster20233dimensionalmirrorsymmetry} for the familiar examples of Coulomb branches.

There is a projection map $\pi: \cR \rightarrow \Gr$. For a cocharacter $\lambda: \mathbb{C}^* \to T \subset G$, an orbit of $G_\cO$ action in $\Gr$ which is closed for cominuscule $\la$ is denoted as $\Gr^\lambda =G_{\mathcal{O}} \cdot t^\lambda$. The preimage of $\Gr^\lambda$ is denoted as $\cR^\la$. See the commutative diagram: 
\[
\begin{tikzcd}
\mathcal{R} \arrow[d] 
  & 
\mathcal{R}^{\lambda} \arrow[d] \arrow[hook]{l}  \\
\mathrm{Gr}  & 
\mathrm{Gr}^{\lambda} \arrow[hook]{l}
\end{tikzcd}
\]
 
In this construction, the monopole operator $r^\lambda$ is defined as the fundamental class $[\cR^\la]$.

\subsection{Abelian Coulomb Branch Algebras and Localization} \label{sec:localization}

Let $\mathfrak{g}$ be the Lie algebra of the reductive group $G$. Let $T \subset G$ be a maximal torus with Lie algebra $\mathfrak{t} \subset \mathfrak{g}$, and denote by $\mathfrak{t}^*=\operatorname{Hom}_{\mathbb{C}}(\mathfrak{t}, \mathbb{C})$ its dual vector space.

Consider the symmetric algebra on $\mathfrak{t}^*$, defined as
$$
\operatorname{Sym}\left(\mathfrak{t}^*\right):=\bigoplus_{k=0}^{\infty} \operatorname{Sym}^k\left(\mathfrak{t}^*\right),
$$
which can be identified with the algebra of polynomial functions on $\ft$.
The degree zero part of $H_T^*(\mathrm{pt})$ is $\operatorname{Sym}\left(\mathfrak{t}^*\right)$.
By \cite{braverman2018coulomb}, for $T$ a torus and $Y$ its coweight lattice, we have an isomorphism of $\operatorname{Sym}(\ft^*)$-modules:
\begin{equation} \label{eq:abeliancoulomb}
    H_*^{T_{\mathcal{O}}}(\mathcal{R})=\bigoplus_{\lambda \in Y} \operatorname{Sym}\left(\mathfrak{t}^*\right) r^\lambda.
    \end{equation}

There exists the Cartan subalgebra $\Theta \cong \mathrm{Sym}(\mathfrak{t}^*)$  of the Coulomb branch algebra $\cA_\hbar (G, \bN)$, the commutative subalgebra generated by the gauge-invariant polynomials in $\ft^*$, see \cite[3(vi)]{braverman2019mathematical}. The monopole operators in $\cA_\hbar (G,\bN)$ commute with the Cartan subalgebra. 

Let \( G \) be a reductive group with maximal torus \( T \). 
By \cite[Lemma 5.10]{braverman2019mathematical}, there is an algebra homomorphism \begin{equation} \label{eq:2.3}
 \iota_*: H_*^{T_{\mathcal{O}} \rtimes \CC^*}\left(\mathcal{R}_{T, \mathbf{N}_T}\right)^W \rightarrow H_*^{G_{\mathcal{O} } \rtimes \CC^* }(\mathcal{R}) .
\end{equation}
According to \cite[Remark 5.23]{braverman2019mathematical}, the map $\iota_*$ becomes an isomorphism 
 if we invert the expressions $\hbar, \beta+m \hbar$ where $\beta$ is a root of $G$ and $m$ is an integer, considered as elements in $H_{T \times \mathbb{C}^{\times}}^*(\mathrm{pt}).$
For an algebra $A$, we denote $A _{\text{loc}}:= A [ \hbar^{-1},( \beta + a \hbar)^{-1} ]_{\beta\in \Delta, a \in \ZZ}$.

The same remark also shows that the following diagram commutes.

\begin{equation}
\begin{tikzcd} 
H_*^{T_{\mathcal{O}} \rtimes \CC^*}\left(\mathcal{R}_{T, \mathbf{N}_T}\right)^W \arrow[r, hook, "\iota_*"] \arrow[d, hook] 
& H_*^{G_{\mathcal{O} } \rtimes \CC^* }(\mathcal{R}) \arrow[d, hook] \arrow[ld, hook, "{\ab}"'] \\
 H_*^{T_{\mathcal{O}} \rtimes \CC^*}\left(\mathcal{R}_{T, \mathbf{N}_T}\right)^W \Sinv \arrow[r, "\cong"] 
& H_*^{G_{\mathcal{O} } \rtimes \CC^* }(\mathcal{R})  \Sinv
\end{tikzcd}
\end{equation}


 Given any representation of \( \cA_\hbar (T, \bN )^W [ \hbar^{-1},( \beta + a \hbar)^{-1} ]_{\beta\in \Delta, a \in \ZZ} \), we can compose with the homomorphism $\ab$ to obtain a representation of  \( H_*^{G_{\mathcal{O} } \rtimes \CC^* }(\mathcal{R})  \).

The map $\ab$ has a particularly nice description on the monopole operators corresponding to minuscule coweights $\lambda$, applying the Atiyah-Bott fixed point theorem to $\operatorname{Gr}^{\lambda}$. 
By \cite[Proposition A.2]{braverman2018coulomb}, we have an equality 
\begin{equation} \label{eq:Vlambda}
 \bV_{\lambda} =  \sum_{w\in W/W_{\lambda}}\frac{r^{w\lambda}}{\displaystyle \prod_{\substack{\beta \in \Delta ^+\\\langle \lambda, \beta\rangle \neq 0}}w\beta}.
\end{equation}
\begin{Lemma} \label{363}
There is an injective map of algebras $ R: \cA_\hbar (G, \bN) \hookrightarrow \cA_\hbar(G,0)$. The operator representation of $\cA_\hbar (G, \bN)$ is a restriction of that of $\cA_\hbar(G,0)$.
\end{Lemma}

We note the compatibility of the diagram with the inclusion into the pure gauge theory. 
After inverting $\hbar$ and $\beta+m \hbar$, this inclusion becomes an isomorphism, so there are injective maps $H^{T_\cO \times \CC^* }( \pt) \rightarrow H^{T_\cO \times \CC^* }( \pt)_{\text{loc}}$ and $H^{G_\cO \times \CC^*}(\pt) \rightarrow H^{G_\cO \times \CC^*}(\pt)
_{\text{loc}} $ that make the following diagram commute.

\begin{equation} \label{368}
\begin{tikzcd}
H_*^{T_{\mathcal{O}} \rtimes \CC^*}\left(\mathcal{R}_{T, \mathbf{N}_T}\right)^W \arrow[r, hook, "\iota_*"] \arrow[d, hook] 
& H_*^{G_{\mathcal{O}} \rtimes \CC^*}(\mathcal{R}) \arrow[d, hook]  \\
H_*^{T_{\mathcal{O}} \rtimes \mathbb{C}^*}\left(\mathcal{R}_{T, 0}\right)^W \arrow[d] 
& H_*^{G_{\mathcal{O}} \rtimes \CC^*} (\cR_{G,0}) \arrow[d] \\
H_*^{T_{\mathcal{O}} \rtimes \CC^*}\left(\mathcal{R}_{T, 0 }\right)^W  \Sinv \arrow[r, "\cong"] 
& H_*^{G_{\mathcal{O}} \rtimes \CC^* }(\mathcal{R} _ {G,0}) \Sinv
\end{tikzcd}
\end{equation}

In particular, we have inclusions \begin{equation} \label{eq:G0N}
    \cA_{\hbar} (G, \bN) \hookrightarrow  \cA_{\hbar} (G, 0 ) \hookrightarrow \cA_{\hbar} (G, 0) \Sinv .
\end{equation}

 We think about \cref{eq:Vlambda} as an equality in $\cA_{\hbar} (G, 0) \Sinv$ relating the abelian and non-abelian monopole operators. 
 
The set $P_{G, \mathbf{N}}$ is defined as the collection of all elements in $\mathcal{A}_{\hbar}(G, \mathbf{N})$ that are in the image of the natural map from $H^{G_{\mathcal{O}} \times \mathbb{C}^*}(\mathrm{pt})$. These elements correspond to $W$-invariant polynomials on the Cartan subalgebra $t$, i.e., elements of $\operatorname{Sym}\left(t^*\right)^W$. This is consistent with the identification of the Cartan subalgebra $\Theta \cong \operatorname{Sym}\left(t^*\right)^W$ within $\mathcal{A}_{\hbar}(G, \mathbf{N})$.



\section{Operator Representation} \label{sec:operrep}

In this section, we provide an operator representation of the quantized Coulomb branches, on an inner space of functions related to the Gamma function.

\subsection{Analysis} \label{sec:analysis}
\subsubsection{}
We will frequently use the following functional equations for the $\Gamma$-function:
\begin{equation} \label{funcgamma}
    \Gamma(z+1) = z \Gamma(z).
\end{equation}

\begin{equation}\label{eq:gammacosh}
    \Gamma ( \frac{1}{2} + i \sigma )   \Gamma ( \frac{1}{2} - i \sigma ) = \frac{\pi}{ \cosh{ \pi \sigma} }.
\end{equation}

The $P- \Gamma$ function $f_b\colon \RR \to \CC $ for an integer  $b\in \ZZ$ is defined  by the formula 
\begin{equation} 
f_b (\sigma )= 
p_b (\sigma) \Gamma \left( \frac{1+ 2b}{2}  + i \sigma \right),  \text{where }    p_b (\sigma)= \prod_{k=0}^{b-1} \left( \frac{2k+1}{2} +i \sigma \right) .
\end{equation}
Here, the domain of $f_b$ parameterizes the line $ Re(z)=\frac{1}{2}+b$.

We consider meromorphically extending $f_b$.
For each integer \( b \in \mathbb{Z} \), define the meromorphic function \( \tilde{f}_b: \mathbb{C} \to \mathbb{C} \) in terms of \( f_b: \mathbb{R} \to \mathbb{C} \) via the coordinate transformation  :  
\begin{equation}  \label{eq:fd}
\tilde{f}_b(z) = f_b\left(-i(z - b - \tfrac{1}{2})\right)= \Gamma (b + \frac{1}{2} + i \sigma) \cdot \prod_{i= 1}^{-b} (b + \frac{1}{2} - i + i \sigma). \quad \forall z \in \mathbb{C}.  
\end{equation}  
Equivalently, for real \( \sigma \in \mathbb{R} \), the function \( f_b \) is expressed as:  
\begin{equation}  
f_b(\sigma) = \tilde{f}_b\left(b + \tfrac{1}{2} + i\sigma\right).  
\end{equation}  
That is, the function $f_b$ is the pullback of $\tilde{f}_b$ by the identification of $\RR$ with an imaginary line.

Note that $f_d$ is a Schwartz function.
\subsubsection{}
\begin{Definition}  \label{def:A}
Let \( \A \subset \mathbb{C} \) be the open vertical strip centered at \( \Re(z) = \frac{1}{2} + d \), with width \( 2(1 + \epsilon) \), which is the set 
\[
\A = \left\{ z \in \mathbb{C} \,\bigg|\, \frac{1}{2} + d - (1 + \epsilon) < \Re(z) < \frac{1}{2} + d + (1 + \epsilon) \right\}.
\]  
Equivalently, this can be written as:  
\[
\A = \frac{1}{2} + d + (-1 - \epsilon, 1 + \epsilon) + i\mathbb{R},
\]  
where the boundaries are the vertical lines \( \Re(z) = \frac{1}{2} + d \pm (1 + \epsilon) \), extending infinitely in the imaginary direction. 
\end{Definition}  

\begin{Lemma}
The meromorphic function $\tilde{f}_d$ is holomorphic on the strip $\A$.
\end{Lemma}
\begin{proof}
 On the extended strip $\frac{1}{2} + d \pm (1 + \epsilon)i\mathbb{R}$, the $\Gamma$ functions are holomorphic except for simple poles at $z=d$ and $z=d+1$ when these are negative integers. However, the $p_{d}$ term in the 
  $P-\Gamma$ function precisely cancels these poles.
This completes the proof.
\end{proof}

\begin{Definition} \label{def:P} 
  Consider the set $\RR \times \ZZ $, identified with a subset of $\CC$ by the injective map $(a,b) \mapsto ia+b $.
   We let $\mathcal{P} (\RR \times \ZZ)$ be 
  the span (finite linear combinations) of functions on $\RR \times \ZZ $
  \begin{equation}
   g_{b',n}(\sigma,b) = (1/2+b +i\sigma)^n \delta_{b,b'}f_{b}(\sigma)
  \end{equation}
  for all $n\in \ZZ_{\geq 0}$, $b'\in\ZZ$ and $\sigma \in \RR$.
\end{Definition}
\subsubsection{}

\begin{Proposition} \label{prop:innerprod1}
The space \(\mathcal{P}(\mathbb{R} \times \mathbb{Z})\) admits a well-defined inner product via the Haar measure on \(\mathbb{R} \times \mathbb{Z}\). Explicitly, for \(f, g \in \mathcal{P}(\mathbb{R} \times \mathbb{Z})\), the inner product is:  
\begin{equation} \label{eq:pairing}
\langle g| f \rangle = \sum_{n \in \mathbb{Z}} \int_{\mathbb{R}}  f(x, n) \, \overline{g(x, n)} \, dx,  
\end{equation}
since the Haar measure is the product of the Lebesgue measure \(dx\) on \(\mathbb{R}\) and the counting measure on \(\mathbb{Z}\).  
\end{Proposition}
A standard application of the Weierstrass approximation theorem shows that \(\mathcal{P}(\mathbb{R} \times \mathbb{Z})\) is dense in \(L^2(\mathbb{R} \times \mathbb{Z})\) under the topology induced by this inner product.

In fact, the following generalization of \Cref{prop:innerprod1} holds: the Haar measure on $\FP^n$ (defined in \Cref{3.18}) is given by the usual Lebesgue measure on $\RR^n$ and the counting measure on $\ZZ^n$, as
\begin{equation} \label{eq:GNprd}
\langle g| f\rangle= \sum_{n \in \mathbb{Z}^m} \int_{\mathbb{R}^m} f(x, n) \overline{g(x, n)} d x.
\end{equation}

\subsubsection{} \label{3.1.4}
We consider the following operators acting on the space $\mathcal{P}(\mathbb{R} \times \mathbb{Z})$, where all functions extend holomorphically to strips, which do not hold for arbitrary elements of $L^2 (\RR \times \ZZ)$. The operators involve shifts in the complex variable $\sigma$ and the discrete variable $b$, which we define rigorously below. 

For a function $\psi(\sigma, b) \in \mathcal{P}(\mathbb{R} \times \mathbb{Z})$, the shift $\sigma \mapsto \sigma \pm \frac{i}{2}$ is implemented via analytic continuation. Specifically, since $\psi(\cdot, b)$ is analytic on the strip $\A$ for each fixed $b$, the strip $\A$ is enough for us to perform the following extension. We extend $\psi(\sigma, b)$ to $\sigma \pm \frac{i}{2}$ by expanding it in a power series around $\sigma$ and substituting the shifted variable. This continuation is valid provided the function remains analytic within a neighborhood containing the line segment from $\sigma$ to $\sigma \pm \frac{i}{2}$.
If this function is analytic, we let $\cT_{\pm}$ be
\[
\mathcal{T}_{\pm} \psi(\sigma, b) = \psi\left(\sigma \pm \frac{i}{2}, b\right).
\]
It will also sometimes be useful to use this notation even if $\psi(\sigma,b)$ only has a meromorphic continuation with isolated poles.  The formula $\cT_+$ above defines a linear map  $\mathcal{P}(\mathbb{R} \times \mathbb{Z}) \rightarrow \mathcal{P}(\mathbb{R} \times \mathbb{Z})$.  
We can still interpret $T_-f_d$ as a meromorphic function, but due to the placement of poles, it might not be analytic. For example, $\cT_- f_0$ has a simple pole at $(\si, b) = (0,0) $.   In particular, it does not preserve  $\mathcal{P}(\mathbb{R} \times \mathbb{Z})$.  We will define operators below where we multiply by additional polynomial factors so that this subspace is preserved.

These operations may seem strange, but they are the image under the Mellin transform naturally densely defined on $L^2(\CC^{\times})$ of the operators of multiplication by the coordinate functions $ x^{\pm 1}$ and their conjugates $ \bar{x}^{\pm 1}$.
Let us emphasize that an element of $\cP(\RR \times \ZZ)$ will be supported on finitely many elements of $\ZZ$, and $\cT_\pm$ changes the support of a function as a set, but does not change its finiteness.

\subsubsection{}
Define \textbf{shift operators} over $\cP ( \RR \times \ZZ) $
\begin{align} \label{eq:v}
(v \psi)(\sigma, b) &=\mathcal{T}_{+}[\psi(\sigma, b-1)]=\psi\left(\sigma+\frac{i}{2}, b-1\right) .\\\label{eq:v2}
(\tilde{v} \psi)(\sigma, b) &=\mathcal{T}_{-}[\psi(\sigma, b-1)]=\psi\left(\sigma-\frac{i}{2}, b-1\right) .
\end{align}

By \Cref{def:P}, we can imagine \cref{eq:v} \cref{eq:v2} as meromorphically extending from the line $\Re(z)=d$ and then shifting the complex plane half a unit to the left or right.

Using $\mathcal{T}_{\pm}$, Gaiotto defined \textbf{monomial-shift operators} $X, P, \widetilde{X}, \widetilde{P}$ to combine these shifts with adjustments to $b$ and multiplicative coefficients \cite[Section 5.2]{GaiottoTempFile}:
\begin{align} 
\label{eq:Xop}  X \psi(\sigma, b) &= \mathcal{T}_{+} \psi(\sigma, b-1) = \psi\left(\sigma + \frac{i}{2}, b-1\right), \\
\label{eq:Pop}  P \psi(\sigma, b) &= \left(\frac{1}{2} + i\sigma + \frac{1}{2}b\right) \mathcal{T}_{-} \psi(\sigma, b+1) = \left(\frac{1}{2} + i\sigma + \frac{1}{2}b\right) \psi\left(\sigma - \frac{i}{2}, b+1\right), \\
 \label{eq:Xtildeop} \widetilde{X} \psi(\sigma, b) &= \left(\frac{1}{2} + i\sigma - \frac{1}{2}b\right) \mathcal{T}_{-} \psi(\sigma, b-1) = \left(\frac{1}{2} + i\sigma - \frac{1}{2}b\right) \psi\left(\sigma - \frac{i}{2}, b-1\right), \\
 \label{eq:Ptildeop} \widetilde{P} \psi(\sigma, b) &= \mathcal{T}_{+} \psi(\sigma, b+1) = \psi\left(\sigma + \frac{i}{2}, b+1\right).
\end{align}
Therefore,
\begin{subequations} \label{eq:XPPX}
\begin{align}
X P \psi(\sigma, b) &= \left(-\frac{1}{2} + i \sigma + \frac{1}{2} b \right) \psi(\sigma, b), \label{eq:XP} \\
P X \psi(\sigma, b) &= \left(i \sigma + \frac{1}{2} b + \frac{1}{2} \right) \psi(\sigma, b), \label{eq:PX} \\
\widetilde{X} \widetilde{P} \psi(\sigma, b) &= \left(i \sigma - \frac{1}{2} b + \frac{1}{2} \right) \psi(\sigma, b), \label{eq:tXtP} \\
\widetilde{P} \widetilde{X} \psi(\sigma, b) &= \left(-\frac{1}{2} + i \sigma - \frac{1}{2} b \right) \psi(\sigma, b). \label{eq:tPtX}
\end{align}
\end{subequations}

      For \begin{equation} \label{eq:1}
             |1 \rangle = \delta_{b,0} \Gamma (\frac{1}{2} + \frac{b}{2} - i \sigma ) , \end{equation}
             under the definition of (\ref{eq:Xop}--\ref{eq:1}), the following holds:
\begin{align} \label{eq:Xm}
    X^m |1 \rangle &= \delta_{b,m} \Gamma \left( \frac{1}{2} + \frac{b}{2} - i \sigma \right) = \tilde{f}_{-m}, \\ \label{eq:Pn}
    P^n |1 \rangle &=  \delta_{b, -n} \Gamma \left( \frac{b}{2} + \frac{1}{2} - i \sigma \right) \cdot \prod_{k=1}^{n} \left( i \sigma + \frac{b}{2} + \frac{2k-1}{2} \right) = \tilde{f}_n.
    \end{align}

\begin{align}
    XP |1\rangle &= 
\frac{1}{2} \left(-1 + b + 2i\sigma\right) \Gamma\left(\frac{1}{2}\left(1 + b - 2i\sigma\right)\right) \delta_{0, b}
\\
    PX|1 \rangle & = 
\frac{1}{2} \left(1 + b + 2i\sigma\right) \Gamma\left(\frac{1}{2}\left(1 + b - 2i\sigma\right)\right) \delta_{0, b}
\end{align}
Therefore, defining $\mu = :XP : = XP+ \frac{1}{2} = PX - \frac{1}{2}$, we have that 
\begin{equation} \label{611}
    \mu |1 \rangle = 
   \frac{1}{2}(b+2 i \sigma) \Gamma\left(\frac{1}{2}(1+b-2 i \sigma)\right) \delta_{0, b}.
\end{equation}

The equation \cref{611} tells us that $\mu$ is multiplication by the complex coordinate we were using to embed $\RR \times \ZZ$ into $\CC$.

\subsubsection{}
 

\begin{Lemma}
 The operators $X, P, \widetilde{X}, \widetilde{P}$, defined explicitly in equations \eqref{eq:Xop}--\eqref{eq:Ptildeop}, acting as monomial-shift operators, map $\mathcal{P}(\mathbb{R} \times \mathbb{Z})$ into itself.
\end{Lemma}
\begin{proof}
By Definition \ref{def:P}, it suffices to compute  \eqref{eq:Xop}--\eqref{eq:Ptildeop} acting on $\delta_{b,b'}p (\si) f_b(\sigma)$ for $p(\si)$ an arbitrary polynomial.
We compute
\begin{equation}
\begin{aligned}
& X\left(\delta_{b, b} p(\sigma) f_b(\sigma)\right)=\delta_{b, b^{\prime}+1} p\left(\sigma+\frac{i}{2}\right) f_b\left(\sigma+\frac{i}{2}\right) \\
& P\left(\delta_{b, b} p(\sigma) f_b(\sigma)\right)=\delta_{b, b^{\prime}-1}\left(i \sigma+\frac{1}{2} b^{\prime}\right) p\left(\sigma-\frac{i}{2}\right) f_b\left(\sigma-\frac{i}{2}\right) \\
& \widetilde{X}\left(\delta_{b, b} p(\sigma) f_b(\sigma)\right)=\delta_{b, b^{\prime}+1}\left(i \sigma-\frac{1}{2} b^{\prime}\right) p\left(\sigma-\frac{i}{2}\right) f_v\left(\sigma-\frac{i}{2}\right) \\
& \widetilde{P}\left(\delta_{b, b} p(\sigma) f_b(\sigma)\right)=\delta_{b, b^{\prime}-1} p\left(\sigma+\frac{i}{2}\right) f_{b^{\prime}}\left(\sigma+\frac{i}{2}\right)
\end{aligned}
\end{equation}
and we note that each of the functions lies in $\cP (\RR \times \ZZ) $.
Thus, \(X, P, \widetilde{X}, \widetilde{P}\) maps \(\mathcal{P}(\mathbb{R} \times \mathbb{Z})\) to itself.
\end{proof}

\subsection{Abelian Pure Gauge Theory}
We use $\mathcal{F} \cP (\RR \times \ZZ)$ to denote the space of functions on $ \RR \times \ZZ $ which extend to meromorphic functions with isolated poles in the complex plane, as discussed in \Cref{3.1.4}.

\begin{Lemma} \label{lemma:pureab}
    For $G = \CC^*$ with trivial $\bN$, we have a holomorphic operator representation of the quantized Coulomb branch algebra $\mathcal{A}_\hbar(G, 0)$ on $\mathcal{FP}(\RR \times \ZZ)$ by the following map $ \alpha_C$:
    \begin{align}
        r &\mapsto v,
    \end{align}
    where $t$ is a generator of $\CC[\ZZ]$. 
    with \( z := \frac{1}{2} - i\sigma + \frac{1}{2} b \). The shift operator $v$ is defined by \cref{eq:v}.
    Similarly, we have an anti-holomorphic operator representation by the following map $\bar{\alpha}_C$:
    \begin{align}
        r &\mapsto - \widetilde{v}  .
    \end{align}
\end{Lemma}
We comment on the connection to the GKLO representation. The actions are defined using the same embedding of $\cA_\hbar (G, \bN)$ into difference operators as \cite[A.1]{braverman2019mathematical}, but incorporating two different actions on the ring of difference operators on the functions on $\RR \times \ZZ$, which come from holomorphic and anti-holomorphic continuation of functions.
\begin{proof}
By \cite[Example 7.2]{teleman2019rolecoulombbranches}, 
to show that the representations are well-defined, it is enough to show the following holds:
    \begin{equation}
u v=v(u+1) \qquad\qquad \widetilde{u} \widetilde{v}=\widetilde{v}(-\widetilde{u}+1)
\end{equation}
for multiplication operators $u=i \sigma+\frac{1}{2} b$ and $\widetilde{u}=i \sigma-\frac{1}{2} b$. 

We check the first relation holds:\begin{equation}
\begin{split}
[v((u+1) \psi)](\sigma, b) 
&= \left(i \left(\sigma + \frac{i}{2} \right) + \frac{1}{2} (b - 1) + 1\right) \\
&\quad \times \psi\left(\sigma + \frac{i}{2}, b - 1\right) \\
&= \left(i \sigma + \frac{1}{2} b\right) \psi\left(\sigma + \frac{i}{2}, b - 1\right) \\ &=[u(v \psi)](\sigma, b).
\end{split}
\end{equation}
The second one is done similarly.
\end{proof}

For larger abelian $T =( \CC^*)^n $ with trivial matter, the $n$-fold direct sum of the representations of $\cA ( \CC^* , 0) $ gives a representation of $ \cA ((\CC^*)^n , 0 ) $. For $T^n = (\mathbb{C}^\times)^n$ with cocharacter lattice $\mathbb{Z}^n$, the weights of the standard representation $\mathbb{C}^n$ are the coordinate projections $\xi_i : \mathbb{Z}^n \to \mathbb{Z}$, $\xi_i(\lambda) = \lambda_i$.

We denote the operators acting on the $i$-th coordinate as  $v_i$, where each $v_i$ is given by \eqref{eq:v}. 

\begin{Lemma} \label{lemma:pureTrep}
There is a representation of $\cA (T^n, 0)$ on $\cP (\RR \times \ZZ)^n$ 
   We call this the {\bf holomorphic operator representation}.
  The representation is given by 
    \begin{equation} \label{eq:3.62}
    r^\lambda \mapsto  \prod_i 
  v_i^{ \xi_i (\lambda )}
     \end{equation}
     for characters $\xi_i$ and coweights $\lambda$.
  The antiholomorphic operator representation can be expressed similarly by   \begin{equation}
    r^\lambda \mapsto  \prod_i 
 - \tilde{v}_i^{ \xi_i (\lambda )}.
     \end{equation}
     Here, the operators $\tilde{v}_i$ are \eqref{eq:v} on the $i$-th coordinate.
\end{Lemma}

\subsection{Representation of $G = \CC^*, \bN = \CC $ case}\label{sec:C}

For $G = \CC^*, \bN = \CC$, we have the classical Coulomb branch algebra as generated by the monopole operators $r^e, r^{-e}$ with the relation $ [r^e,r^{-e}]=1$,  and the polynomials in $t^n \in \CC[t]$, with $t=\frac{1}{2}\left(r^e r^{-e}+r^{-e} r^e\right)$.
By checking the degree given by $\lambda$, we have that the following.

We define \begin{equation} \label{667}
    z := \frac{1}{2} - i\sigma + \frac{1}{2} b .
\end{equation} 

\begin{Lemma} \label{lemma:coulomb_representation}
    For $G = \CC^*$ and $\bN = \CC$, we have a holomorphic operator representation of the quantized Coulomb branch algebra $\mathcal{A}_\hbar(G, \bN)$ on $\mathcal{P}(\RR \times \ZZ)$ by the following map $ \alpha_C$:
    \begin{align}
        r^e &\mapsto X, \\
        r^{-e} &\mapsto P,
    \end{align}
    where $t$ is a generator of $\CC[\ZZ]$. In particular, the following relations hold:
  \begin{equation}
       r^e r^{-e} \mapsto \left(z - \frac{1}{2}\right), \qquad\qquad
        r^{-e} r^e \mapsto \left(z + \frac{1}{2}\right), \qquad\qquad t \mapsto z.
  \end{equation}

    We have an antiholomorphic operator representation by the following map $\bar{\alpha}_C$:
    \begin{align}
        r^e \mapsto - \widetilde{P}  , \qquad\qquad  r^{-e} \mapsto -\widetilde{X}.
    \end{align}
\end{Lemma}

\begin{proof}
By \eqref{eq:XPPX}, the maps $\alpha_C, \bar{\alpha}_C$ preserve the relation $[r^e,r^{-e}]=1$.
\end{proof}

Consider the holomorphic operator representation $\alpha_C $ and antiholomorphic operator representations $\bar{\alpha}_C$ of $\cA_\hbar (\CC^* , \CC)$ defined in \ref{sec:C} as:
      \begin{align}
    \label{eq:2.16}    X &= \alpha_C \left( r^e \right), \quad  
        P = \alpha_C \left( r^{-e} \right), \\ \label{eq:2.17}
        \widetilde{X} &= -\bar{\alpha}_C \left( r^{-e} \right), \quad  
        \widetilde{P} = -\bar{\alpha}_C \left( r^e \right).
    \end{align}

    We let $\langle 1 |$ denote \textbf{the adjoint of $|1\rangle$ }, the unique linear functional such that for $|1 \rangle \in \cP(\RR \times \ZZ)$, for $\langle \bullet, \bullet \rangle$ defined in \eqref{eq:pairing}, for any vector $|v\rangle$, \begin{equation}
    \langle 1| : |v \rangle \mapsto \langle 1| v \rangle.\end{equation}
Since $\cT_\mp$ is adjoint to $\cT_\pm$ and multiplication by $p(z)$ is adjoint to $p(\bar{z})$, we have:

\begin{Lemma} \label{lemma:alphaC1}    
  \begin{align} \label{eq:2.18}
\alpha_C\left(r^{e}\right)|1\rangle&=\bar{\alpha}_C\left(r^{-e}\right)|1\rangle, & \alpha_C\left(r^{-e}\right)|1\rangle&=\bar{\alpha}_C\left(r^{e}\right)|1\rangle, \\
            \langle 1| \alpha_C\left(r^{e}\right)&=-\langle 1| \bar{\alpha}_C\left(r^{-e}\right), & \langle 1| \alpha_C\left(r^{-e}\right)&=-\langle 1| \bar{\alpha}_C\left(r^{e}\right).   \label{eq:2.19}
        \end{align}
\end{Lemma}
\begin{proof}
The operator \cref{eq:Xtildeop} acts on \cref{eq:1} and gives 
\begin{equation}
  - \widetilde{X} |1\rangle= - \left(\frac{1}{2}+i \sigma-\frac{1}{2} b\right) \delta_{b-1, 0} \Gamma\left(\frac{1}{2}+\frac{b-1}{2}-i \sigma\right) = \delta_{b, 1} \Gamma\left(\frac{1}{2}+\frac{b}{2}-i \sigma\right) = X  |1\rangle
\end{equation}
The second equality holds by \cref{funcgamma}. Additionally, by a similar computation for $P$ and $\widetilde{P}$, the following holds: 
    \begin{equation} \label{eq:XXPP1}
            X|1\rangle=-\widetilde{X}|1\rangle, \qquad P|1\rangle=-\widetilde{P}|1\rangle .
            \end{equation}
From \crefrange{eq:2.16}{eq:2.17}, and \cref{eq:XXPP1}, we obtain \cref{eq:2.18}.
\end{proof}

\subsection{The Pontryagin Dual}
For a locally compact abelian group \(G\), its Pontryagin dual \(\widehat{G}\) is the group of continuous characters \(\chi: G \to \mathrm{U}(1)\). For a standard treatment of Pontryagin duality, see \cite[Chapter 12]{rudin1991_fag}. In particular, for \(G = \mathbb{C}^\times\) and its decomposition
\(\mathbb{C}^\times \cong \mathbb{R} \times S^1,\)
the Pontryagin duality is given by the following.
Every $U(1)$ valued character of $\RR \times S^1$ is of the form $\chi_{x,n}(\alpha,\theta)=e^{2\pi i\alpha x}\theta^n $; this shows that $\widehat{(\mathbb{R} \times S^1)}  \cong  \mathbb{R} \times \mathbb{Z}$.  Dually, every $U(1)$ valued character of $\RR \times \ZZ$ is of the form $\chi_{\alpha,\theta}(x,n)$ defined by
$\chi_{\alpha,\theta}(x,n)=e^{2\pi i\alpha x}\theta^n$. 
Note that this is the same formula again, but now we have switched the role of the elements of $\RR \times S^1$ and $\RR \times \ZZ$.

By Pontryagin duality, the trace formula below can also be interpreted as a natural integral over $\CC^*$. 
\begin{Remark}
As shown in \cite[3.1.1]{GaiottoTempFile}, the function 
     $ |1 \rangle = \delta_{b,0} \Gamma (\frac{1}{2} - i \sigma)$
 arises as the multiplicative Fourier transform (i.e. the Fourier transform with respect to the multiplicative group $\CC ^ \times $)  of the Gaussian $e^{-|z|^2}$. The multiplicative Fourier transform sends multiplication by $z$ and $\bar {z}$ to the actions $v_i$ and $\tilde{v}_i$.
 \end{Remark}

\subsection{Twisted Trace}

The algebra $\cA_\hbar(G,\bN)$ is the Coulomb branch algebra of the theory with $G = \CC^*$ and $N = \CC$, and there is an operator representation as shown in Section \ref{sec:C}.
\begin{Lemma}
We can define a twisted trace for polynomial $w$'s on $\cA_\hbar(G,\bN)$ via the formula:
\begin{equation} \label{345}
\operatorname{Tr} = \langle 1| g \alpha_C(w)|1\rangle = \langle 1| \alpha_C(w) g|1\rangle.
\end{equation}
The twisted trace on $P_{G,\bN}$ is given by
\begin{equation} \label{eq:trace1}
    \Tr(R (\sigma))=\int_{\mathbb{R}} e^{2 \pi i \zeta(\sigma)} R(\sigma) \mathrm{w} \mathrm{d} \sigma,
\end{equation}
    where
    $$
    \mathrm{w}(\sigma)=\frac{\pi}{\cosh (\pi \sigma)},
    $$
    and $ R(\sigma)$ is a polynomial in $\sigma$.
\end{Lemma}
 This integral twisted trace is a special case of the general classification given by Etingof, Klyuev, Rains and Stryker. Up to a constant, this is \cite[Example 3.4]{Etingof2021Twisted}.
\begin{proof}

  The ground state is 
     \begin{equation}
        |1 \rangle = \delta_{b,0} \Gamma \left( \frac{1}{2} - i \sigma \right).
     \end{equation}
Then the trace $\Tr(1)$ is a polynomial times the product of functions 
\begin{equation} \label{eq:11}
  \CH   = \frac{\pi}{\cosh(\pi \sigma)}.
\end{equation}
   For a word in $r^{1}$ and $r^{-1}$ to have a nonzero trace, it must have the same number of $r^{1}$'s and $r^{-1}$'s appearing. 
   The span of such words is the same as a polynomial in $\mu$.
   It then suffices to consider the trace of $\mu^n$, we have
    \begin{equation}
        \mu^n |1 \rangle = (XP+ \frac{1}{2})^n | 1 \rangle = 
\frac{1}{2^n} ( b + 2i\sigma)^n \, \Gamma\left(\frac{1}{2} (1 + b - 2i\sigma)\right) \, \delta_{0, b}
  \end{equation}
For the operator $\mu^n$, we have 
  \begin{equation}
\begin{split} 
  \langle 1 |  \mu^n |1 \rangle 
    &= \langle 1| \left(XP + \frac{1}{2}\right)^n | 1 \rangle \\
    &= \int_\RR \Gamma\left(\frac{1}{2} (1 + b + 2i\sigma)\right) 
         \frac{1}{2^n} (b + 2i\sigma)^n \\
    &\qquad \times \Gamma\left(\frac{1}{2} (1 + b - 2i\sigma)\right) \delta_{0, b}  d\sigma \\
    &= \int_\RR (i\sigma)^n \frac{\pi}{\cosh(\pi \sigma)}  d\sigma.
\end{split}
\end{equation}
In the second equality, we use the definition of the inner product as an integral \cref{eq:pairing}.  The third equality holds by \cref{eq:gammacosh}.


The equation \cref{345} holds by \Cref{lemma:alphaC1} an analogous proof of \cite[Lemma 2.3]{BenWinter}.
The uniqueness follows by \cite[Theorem 2.2]{BenWinter}.
\end{proof}
Our twisted trace matches the one defined in  \cite[Example 3.4]{Etingof2021Twisted}, where the automorphism   $\rho^2(a) = (-1)^F e^{2 \pi \zeta (a)}  $ is taken to be $ (-1)^F e^{2 \pi i \mathrm{m} }$. The difference between the FI parameter $\zeta$ vs the mass parameter $\mathrm{m}$ comes from identifying this algebra as a Coulomb branch as opposed to a Higgs branch.
   The authors of \cite{Etingof2021Twisted} also show that all twisted traces on the Weyl algebra are the same up to a constant.


\subsection{Faithful Representations}
We consider quantized abelian Coulomb branch algebras  $\mathcal{A}_h(T, \bN) $, where   $T=(\mathbb{C}^*)^n, \mathbf{N}=\mathbb{C}^n$. The maximal torus $T$ acts on $\bN$ by the usual identification with diagonal matrices. The representation space is the space $\mathcal{P}(\mathbb{R} \times \mathbb{Z})^n$ of functions of $n$ independent pairs of variables $(\sigma_k, b_k) \in \mathbb{R} \times \mathbb{Z}$, $k=1, \dots, n$. A general function in this space is of the form $\psi(\sigma_1, \dots, \sigma_n, b_1, \dots, b_n)$.
Let $r^{e_k} $ denote the monopole operator for coweight $\lambda =(0,\cdots,0, 1 , 0,\cdots,0)$, with an $1$ at the $k$-th position, and analogously for $r^{-e_k} $.


       Here, the ground state in this \( n \)-dimensional case is given by:
        \begin{equation} \label{3.52}
            |1\rangle = \delta_{\mathrm{b}, 0} \prod_{k=1}^n \Gamma\left(\frac{1}{2} - i \sigma_k \right),
        \end{equation}
        where \( \mathrm{b} := (b_1, \dots, b_n) \) and \( \mathrm{\sigma} := (\sigma_1, \dots, \sigma_n) \) represent the integer and real parts of \( (\mathbb{R} \times \mathbb{Z})^n \), respectively.

We consider the following operators $X_{k}, \widetilde{X}_{k}, P_{k} \widetilde{P}_{k}$, on the space $\cP(\RR \times \ZZ)^{m}$ which are the operators $X, \widetilde{X}, P, \widetilde{P}$  defined in \eqref{eq:Xop}-\eqref{eq:Ptildeop} acting on the $k$th coordinate.
We identify the following 
\begin{align}
    \begin{aligned} \label{eq:multiXP}
& X_k=\alpha_C\left(r^{e_k}\right), \quad P_k=\alpha_C\left(r^{-e_k}\right), \\
& \widetilde{X}_k=-\bar{\alpha}_C\left(r^{-e_k}\right), \quad \widetilde{P}_k=-\bar{\alpha}_C\left(r^e_k\right) .
\end{aligned}
\end{align}

 We define $X_i$ and $P_i$ by \eqref{eq:Xop} \eqref{eq:Pop} on each coordinate. Therefore, $X_i^d $ and $  P_i^ d $ are defined by their compositions.

 The Cartan subalgebra $\Theta$ is a symmetric function in $z_1, \cdots, z_n$. 
Additionally,\[
[z_k, r^\lambda] = \langle \lambda, e_k \rangle  r^\lambda \quad \text{for any coweight } \lambda.
\]

\begin{Lemma}\label{lemma:coulomb_representation_ndim}
    Consequently, the holomorphic operator representation of $\cA(T^n, \CC^n )$ on $\cP (\RR \times \ZZ)^n$ is given by 
    \begin{equation} \label{798}
    r^\lambda \mapsto    \prod_i 
    \begin{cases}
      X_i^{ \xi_i (\lambda )} \quad \text{if }  \xi_i( \lambda) >0 \\
      P_i^ {\xi_i (\lambda) } \quad \text{if } \xi_i( \lambda) < 0,
    \end{cases}
     \end{equation}
     for characters $\xi_i$ and coweights $\lambda$.
  The antiholomorphic operator representation can be expressed similarly.
\end{Lemma}

By the same reasoning as \Cref{lemma:alphaC1}, we have
\begin{Lemma} \label{lm:alphac}
For each $1\leq k \leq n$,
      \begin{align} 
\alpha_C\left(r^{e_k}\right)|1\rangle&=\bar{\alpha}_C\left(r^{-e_k}\right)|1\rangle, & \alpha_C\left(r^{-e_k}\right)|1\rangle&=\bar{\alpha}_C\left(r^{e_k}\right)|1\rangle, \\
            \langle 1| \alpha_C\left(r^{e_k}\right)&=-\langle 1| \bar{\alpha}_C\left(r^{-e_k}\right), & \langle 1| \alpha_C\left(r^{-e_k}\right)&=-\langle 1| \bar{\alpha}_C\left(r^{e_k}\right).  
        \end{align}
\end{Lemma}

\subsection{Abelian Theory with Nontrivial Flavor} \label{sec:abelian}
We consider a representation of the quantized Coulomb branch algebra \( \mathcal{A}_\hbar (G, \bN) \), \( G = (\mathbb{C}^*)^m \) and \( \bN = (\mathbb{C})^n \), $m<n$. We fix the inclusion of $T^m$ into $T^n$. This inclusion defines the action of the gauge group $G$ on the representation $\bN$.
Let $\left(t_1, t_2, \ldots, t_m\right) \in G$. Then the action on a vector $\left(x_1, x_2, \ldots, x_n\right) \in \mathbb{C}^n$ is given by:
$$
\left(t_1,  \cdots, t_m\right) \cdot\left(x_1, x_2, \cdots, x_n\right)=\left( t_1^{a_{1 }} x_1, \cdots, t_m^{a_{m}} x_m, \cdots,  x_n\right),
$$
where $a_{i } \in \mathbb{Z}$ for $i=1,2, \cdots, m$.

\begin{Definition} \label{coord}
We choose a splitting of the exact sequence on the Lie algebras,
\begin{equation}
    1 \rightarrow \mathfrak{t}^m \rightarrow \mathfrak{t}^n \rightarrow \mathfrak{f} \rightarrow 1,
\end{equation}
and we identify $\ft ^n = \ft^m \oplus \ff$.
We fix $\ms \in \ff$.
Identify $\RR^m$ with the subspace of $\RR^n$ where all but the first $m$ coordinates vanish. We fix a linear isomorphism $\something: \RR^n \to \ft_\RR^n$ compatible with the splitting, in the sense that it induces an isomorphism $\RR^m\to\ft ^m_\RR$.\end{Definition}

We will now change coordinates by $\something$; note that this means that the coordinate functions $\something\left(\sigma_1, \ldots, \sigma_n\right)=(w_1,\dots, w_n)$ are now the weights of the representation, thought of in these new coordinates.


We have an induced map 
\begin{equation}
   ( \ft^n)^* \leftarrow \ff^*,
\end{equation}
and for each $\ms$, we have a maximal ideal $I_\ms \in \ff^*$. 
An action of $F^{\vee}$ on \( \mathcal{A}_\hbar(T^n, \mathbf{N}) \) is a grading by the cocharacter lattice of $F$, so we use the grading induced by the projection from the cocharacter lattice of $T^n$ to $F$. This action preserves the algebraic structure, and the invariants \( \mathcal{A}_\hbar(T^n, \mathbf{N})^{F^{\vee}} \) consist of operators commuting with \( F^{\vee} \).
By \cite[3(vii)]{braverman2019mathematical}, $\mathcal{A}_{\hbar}(G, \mathbf{N})$  is the quantum Hamiltonian reduction of $   \mathcal{A}_{\hbar}(T^n, \mathbf{N}) $ by $ \Fl $:
\begin{Lemma}
\begin{equation}
    \mathcal{A}_{\hbar}(G, \mathbf{N})
    =  \mathcal{A}_{\hbar}(T^n, \mathbf{N})^{\Fl} /I_\ms \cong \End ( \cA _{\hbar}(T^n, \mathbf{N}) / I_ \ms \cA_{\hbar}(T^n, \mathbf{N})).
\end{equation}    
\end{Lemma}

The latter is given by \cite[Section 3.5; Section 5.5]{braden2022quantizationsconicalsymplecticresolutions}
$\mathcal{A}_{\hbar}(T^n, \mathbf{N})$ is graded by the cocharacters of the flavor group. The algebra $ \mathcal{A}_{\hbar}(G, \mathbf{N})$ is in degree zero.
Now we consider the restriction of the operator representation of $   \mathcal{A}_{\hbar}(T^n, \mathbf{N})$
as an operator representation of $\mathcal{A}_{\hbar}(G, \mathbf{N})$. In coordinates, the representation is restricted to the first $m$ components of $\ft^n = \ft^m + \ff$, as a coset, and the cocharacter of $T^m$.

Recall from \Cref{363} that $\cA_\hbar (T, \bN)$ has an injective map to $\cA(T,0)$, so we can pull back the action given in \Cref{lemma:pureab}.  Thus, we have that:

\begin{Lemma} \label{lemma:coulomb_representation_nontrivial_flavor}
For \( G = (\mathbb{C}^*)^m \) and \( \bN = (\mathbb{C})^n \),  the algebra $\cA(G, \bN)$ has a representation on the space $\cP (\RR\times \ZZ)^n$ where monopole operators act by shift operators times polynomials.


\end{Lemma}
The ground state for this representation \ref{lemma:coulomb_representation_nontrivial_flavor} in \( m \)-dimensions, for a fixed nontrivial element $\ms \in \ff$, is given by:
\begin{equation}
|1\rangle=\delta_{\underline{b}, 0} \prod_w \Gamma\left(\frac{1}{2}-i(w(\sigma+\m))\right),
\end{equation}
where $\underline{b} = (b_1, \cdots , b_m) $.
After changing coordinates, this is the pullback of \cref{3.52}.
After the coordinate change and restriction to the first $m$ coordinates, the equations in \Cref{lm:alphac} give the following:
\begin{Lemma} \label{lm:alphacwm}
          \begin{align} 
\alpha_C\left(r^{\lambda}\right)|1\rangle&=\bar{\alpha}_C\left(r^{-\lambda}\right)|1\rangle, & \alpha_C\left(r^{-\lambda}\right)|1\rangle&=\bar{\alpha}_C\left(r^{\lambda}\right)|1\rangle, \\
            \langle 1| \alpha_C\left(r^{\lambda}\right)&=-\langle 1| \bar{\alpha}_C\left(r^{-\lambda}\right), & \langle 1| \alpha_C\left(r^{-\lambda}\right)&=-\langle 1| \bar{\alpha}_C\left(r^{\lambda}\right).  
        \end{align}
\end{Lemma}

\subsection{Nonabelian Gauge Theory}

For the remainder of the paper, we will assume that $G= \prod GL_{a_i}$. 
The twisted traces in the abelian case can be generalized for all quiver gauge theories by localizing $\cA_\hbar (G,\bN) $ to its maximal torus.

\begin{Definition} \label{3.18}
The Weyl group $W \cong S_n$ acts on $\cP (\RR \times \ZZ)^n$ by permuting coordinates:  
\[
w \cdot \big( (\sigma_1, b_1), \dots, (\sigma_n, b_n) \big) = \big( (\sigma_{w(1)}, b_{w(1)}), \dots, (\sigma_{w(n)}, b_{w(n)}) \big).
\]

Let $\FP^n$ be the space of Weyl-invariant functions $f$ where:
\begin{enumerate}
    \item For each fixed $b \in \mathbb{Z}^n$, the map $\sigma \mapsto f(\sigma, \mathbf{b})$ extends to a meromorphic function on $\mathbb{C}^n$.
\item This meromorphic function has no poles on $\mathbb{R}^n$ and decays exponentially as $|\Re(\sigma)| \to \infty$.
    \item For each abelian monopole operator $r^{\lambda}$, the function $ \alpha_C(r^{\lambda})f$ satisfies conditions (1) and (2) above.
\end{enumerate}

\end{Definition}

By the localization \cref{368} and the representation of the abelian theory of \Cref{lemma:coulomb_representation_nontrivial_flavor}, there exists a representation of \( \mathcal{A}_\hbar(G, \mathbf{N}) \), as operators on $\FP^n$.
We denote $z_{i,r} = \frac{1}{2} - i \sigma_{i,r} +\frac{b_{i,r}}{2}$.
We recall the operator \(X_{i,r}, P_{i,r}\) defined in \crefrange{eq:Xop}{eq:Pop}. These shift operators act as:
\begin{align}
X_{i,r}: (\sigma_{i,r}, b_{i,r}) \mapsto (\sigma_{i,r} + i/2, b_{i,r} - 1), \label{3.58} \\
P_{i,r}: (\sigma_{i,r}, b_{i,r}) \mapsto (\sigma_{i,r} - i/2, b_{i,r} + 1). \label{3.59}
\end{align}

We observe that the shifts defined by $X_{i, r}$ and $P_{i, r}$ leave $z_{i, r}=\frac{1}{2}-i \sigma_{i, r}+\frac{b_{i, r}}{2}$ invariant:
under $X_{i, r}$ 
$$
z_{i, r} \mapsto \frac{1}{2}-i\left(\sigma_{i, r}+\frac{i}{2}\right)+\frac{b_{i, r}-1}{2}=\frac{1}{2}-i \sigma_{i, r}+\frac{1}{2}+\frac{b_{i, r}}{2}-\frac{1}{2}=z_{i, r} .
$$
and under $P_{i, r}$:
$$
z_{i, r} \mapsto \frac{1}{2}-i\left(\sigma_{i, r}-\frac{i}{2}\right)+\frac{b_{i, r}+1}{2}=\frac{1}{2}-i \sigma_{i, r}-\frac{1}{2}+\frac{b_{i, r}}{2}+\frac{1}{2}=z_{i, r} .
$$

\begin{Proposition} \label{prop:pureg} \label{prop:general_rep}
For $V_i=\mathbb{C}^{a_i}$, and $\operatorname{GL}(V):=\prod_{i \in Q_0} \operatorname{GL}\left(V_i\right) = \prod_i \GL_{a_i} $ with maximal torus \( T = (\CC^*)^{a_i} \),
we have a holomorphic operator representation $\alpha_C$ of \( \mathcal{A}_\hbar(G, \mathbf{N}) \) on \(\FP ^n \), where for each $V_i$, \( i = 1, \dots, n \), $  t_{i,r} \mapsto  z_ {i,r} $,
and 
the generators $\bV_\lambda$'s defined in 
\eqref{eq:Vlambda} are sent to

 \begin{equation}  \label{eq:opg}
    \alpha_C   \left( \mathbf{V}_\lambda  \right) =\sum_{w \in W / W_\lambda} \frac{  \alpha_C \left(  r^{w \lambda}  \right)  }{\prod_{\substack{\beta \in \Delta^{+} \\\langle\lambda, \beta\rangle \neq 0}} w \beta} = \sum_{w \in W / W_\lambda} \frac{\prod_{i,r} \mathcal{O}_{i,r}^{\xi_{i,r}(w\lambda)}}{\prod_{\substack{\beta \in \Delta^+ \\ \langle \lambda, \beta \rangle \neq 0}} w\beta},
    \end{equation}
    where
    \begin{equation} \label{994}
    \alpha_C \left(  r^{w \lambda}  \right) =\prod_{i,r} \mathcal{O}_{i,r}  = \prod_{i,r}
       \begin{cases}
      X_{i,r}^{ \xi_{i,r} (w \lambda )} \quad \text{if }  \xi_{i,r}( w\lambda) >0 \\
      P_{i,r} ^ {\xi_{i,r} (w\lambda) } \quad \text{if }  \xi_{i,r}( w \lambda) < 0.
    \end{cases}
    \end{equation}

Similarly, for the antiholomorphic representation \(\bar{\alpha}_C\), the operator \(\widetilde{\mathcal{O}}_{i,r}\) corresponding to \(\bar{\alpha}_C \left( r^{w \lambda} \right)\) is defined by replacing \(X_{i,r}\) with \(-\widetilde{X}_{i,r}\) and \(P_{i,r}\) with \(-\widetilde{P}_{i,r}\) in the holomorphic expression:
\[
\bar{\alpha}_C \left( r^{w \lambda} \right) = \widetilde{\mathcal{O}}_{i,r} = \prod_{i,r}
\begin{cases}
\left(-\widetilde{P}_{i,r}\right)^{\xi_{i,r}(w\lambda)} & \text{if } \xi_{i,r}(w\lambda) > 0\\
\left(-\widetilde{X}_{i,r}\right)^{\xi_{i,r}(w\lambda)} & \text{if } \xi_{i,r}(w\lambda) < 0.
\end{cases}
\]
\end{Proposition}

\begin{proof} We provide a proof for one copy of $\GL_{n}$. The proof is similar for a product.


 By the localization isomorphism \(\iota_*\) in \eqref{eq:2.3} and the commutative diagram in \Cref{sec:localization}, the algebra \(\mathcal{A}_h(G,\mathbf{N})\) is isomorphic to the \(W\)-invariant subalgebra of \(\mathcal{A}_h(T,\mathbf{N})[\hbar^{-1}, (\beta + a\hbar)^{-1}]\) via the map \(\ab\). Under this map, the generator \(\mathbf{V}_\lambda\) is expressed as \eqref{eq:Vlambda}.  
  
 The abelian algebra \(\mathcal{A}_h(T,\mathbf{N}_T)\) admits an operator representation on \(\mathcal{P}(\mathbb{R}\times\mathbb{Z})^a\) as constructed in \Cref{sec:abelian}.

We note that the denominator $\prod w\beta$ cancels poles of the numerator under Weyl reflection by the following reasoning.
The denominator $\prod w \beta$ is a product of linear functions $w \beta(\sigma)$, which vanish on hyperplanes $H_{w \beta}=\left\{\sigma \in \mathbb{C}^n: w \beta(\sigma)=0\right\}$. For the overall expression to be regular on $\mathbb{R}^n$, the numerator must vanish on these hyperplanes to cancel the zeros of the denominator.

We take a closer look at the Weyl invariance. The sum over $w \in W / W_\lambda$ is symmetric under the Weyl group. For any $w^{\prime} \in W$, the set $\left\{w^{\prime} w: w \in W / W_\lambda\right\}$ is a reordering of the cosets. The denominator transforms as:
$$
\prod_\beta\left(w^{\prime} w\right) \beta=\operatorname{det}\left(w^{\prime}\right) \prod_\beta w \beta,
$$
where $\operatorname{det}\left(w^{\prime}\right)= \pm 1$. The numerator $\alpha_C\left(r^{w \lambda}\right) f$ transforms under $w^{\prime}$ as:
$$
w^{\prime} \cdot \alpha_C\left(r^{w \lambda}\right) f=\alpha_C\left(r^{w^{\prime} w \lambda}\right) f,
$$
due to the Weyl invariance of $f$ and the definition of $\alpha_C\left(r^{w \lambda}\right)$.

Now we show the cancellation on the hyperplanes. Fix a hyperplane $H=\left\{\sigma: \beta_0(\sigma)=0\right\}$ for some $\beta_0 \in \Delta^{+}$. Consider the terms in the sum corresponding to $w$ and $w^{\prime}$ such that $w^{\prime} \beta_0=-\beta_0$. Then the denominators for $w$ and $w^{\prime}$ differ by a sign on $H$.
The numerators $\alpha_C\left(r^{w \lambda}\right) f$ and $\alpha_C\left(r^{w^{\prime} \lambda}\right) f$ are related by the Weyl action and satisfy:
$$
\alpha_C\left(r^{w^{\prime} \lambda}\right) f=-\alpha_C\left(r^{w \lambda}\right) f \quad \text { on } H,
$$
due to the antisymmetry under reflection across $H$.
Thus, the sum of these two terms vanishes on $H$, canceling the simple zero from the denominator. This argument extends to all hyperplanes $H_{w \beta}$, ensuring that $\alpha_C\left(\mathbf{V}_\lambda\right) f$ has no poles on $\mathbb{R}^n$.

In addition, the space $\FP^n$ ensures that $f$ decays exponentially as $|\Re(\sigma)| \rightarrow \infty$. The operators $\alpha_C\left(r^{w \lambda}\right)$ introduce at most polynomial growth in $\sigma$, but the exponential decay of $f$ dominates. Therefore, $\alpha_C\left(\mathbf{V}_\lambda\right) f$ also decays exponentially, satisfying condition (2) of $E \mathcal{P}$.

While the elementary shift operators $X_{i,r}, P _{i,r}$ do not preserve $\FP^n$ individually, the symmetrized monopole operators $\mathbf{V}_\lambda$ \cref{eq:opg} do preserve this space.

Both numerator and denominator terms are $W$ invariant.

The antiholomorphic representation follows analogously by replacing \(\mathcal{O}_{i,r}\) with \(\widetilde{\cO}_{i,r}\).
\end{proof}

For $\GL_{n}$, the roots can be written as \( \Delta_i = \{ z_r - z_s \}_{r \neq s} \), and the positive roots are \( \Delta_i = \{ z_r - z_s \}_{r > s} \). The obvious extension holds for a product of $\GL_{a_i}$'s.

\subsection{A Subrepresentation}

  We believe the results should hold without changes in other types, but the proof is slightly simpler in this case. We denote the maximal torus of $G$ as $T= \prod\CC^{a_i}$.


We let \( \Delta_{i,+} \) denote the positive roots of \( \operatorname{GL}(V_i) \) (i.e., \( \sigma_r - \sigma_s \) for \( r < s \) within the \( \operatorname{GL}(V_i) \) factor), and \( w \) are weights of the flavor symmetry representation.
\begin{Proposition}  \label{prop:1gauge}
The functions \begin{equation}  \label{eq:1withbeta}
    \IGN : = \delta_{\underline{b}, 0} \frac{\prod_{w} \Gamma\left(\frac{1}{2} - i w(\sigma + \m)\right)}{\prod_{i } \prod_{\beta \in \Delta_{i,+}} \Gamma(i\beta)\Gamma(-i\beta) } \qquad\text{and} \qquad
    |v\rangle : = \delta_{\underline{b}, 0} \frac{\prod_{w} \Gamma\left(\frac{1}{2} - i w(\sigma + \m)\right)}{\prod_{i } \prod_{\beta \in \Delta_{i,+}} \Gamma(i\beta)\Gamma(-i\beta) \beta^2 },
\end{equation} 
are holomorphic on $\A ^{\sum_{a_i}}$ for $\A$ defined in Definition \ref{def:A}.

\end{Proposition}

\begin{proof}
Note that 
\begin{equation}
         \frac{1}{  \Gamma ( i b)\Gamma (- i b) }  \label{eq:gammabete}\qquad\text{and} \qquad
 \frac{1}{ \Gamma ( i b)\Gamma (- i b) b^2 }  =    \frac{1}{  \Gamma ( 1+ i b)\Gamma (1 - i b) }  
\end{equation}
  are entire functions. Thus, the denominators of $\IGN $ and $|v \rangle$ do not have poles.
\end{proof}

\begin{Corollary}
    The functions $\IGN$ and $|v \rangle$ lie in $\FP^{\sum_{a_i}}$.
\end{Corollary}

\begin{Definition}
    For $G$, a product of $\GL(V_i)$, for a fixed nontrivial element \( \m \in \mathfrak{f} \), and some \( a \in \mathbb{Z} \), we define a subrepresentation of Definition \ref{prop:general_rep}, on the \textbf{ground-state-generated space} $\GP$ to be the span (finite linear combinations) of actions of $    \alpha_C   \left( \mathbf{V}_\lambda  \right) $ on
    $|v'\rangle$.
\end{Definition}

We observe that \begin{equation}
\alpha_C(\bV_{-w_0 e} )\alpha_C (V_e) \IGN = -(1+i \sigma_1 + \cdots + i \sigma_n) \cdot|v \rangle , \end{equation}then $|v \rangle \in \GP$.

\section{Twisted Trace on Coulomb branches}  \label{sec:twtr}
The purpose of this section is to prove the following result.

\begin{Theorem} \label{prop:pureg-trace}
For conical theories, the formula
   \begin{equation} \label{eq:GNTrw}
   \Tr(w)= \bign \alpha_C(w) g \IGN =  \bign g \alpha_C (w) \IGN
   \end{equation}
  defines a twisted trace for $\cA_\hbar ( \prod \GL_{a_i} , \bN) $ as operators on the function space $\GP$.

  The twisted trace on polynomials in $\cA_\hbar(G,\bN)$ is defined by the integral:
\begin{equation} \label{eq:trthm}
    \Tr(R (\sigma))=\int_{\mathbb{R}^m} e^{2 \pi i \zeta(\sigma)} R(\sigma) \prod_{Q_0}
     \frac{\prod_{\beta \in \Delta_{+}} (\frac{1}{\pi} \sinh (\pi \beta))^2}{\prod_{w } \frac{1}{\pi} \cosh (\pi w(\sigma+\m))} 
    \mathrm{d} \sigma,
\end{equation}

\end{Theorem}
Before proving this theorem, we will introduce a few lemmas.


We recall the inner product \cref{eq:GNprd} is the natural multi-variable generalization of \Cref{eq:pairing}.
 Building on this, we may generalize \Cref{prop:pureg-trace} by the representation given in \Cref{lemma:coulomb_representation_ndim} and a direct sum of $n$ copies of \eqref{eq:trace1}. This construction yields the following twisted trace on $\cA_\hbar (T, \bN)$, for  $T=(\mathbb{C}^*)^n, \mathbf{N}=(\mathbb{C})^n$:
    \begin{equation}  \label{prop:TnCn}
        \Tr (R(\sigma ))=\int_{\mathbb{R}^n} e^{2 \pi i \zeta(\sigma)} R(\sigma) \mathrm{w}(\sigma)\mathrm{d} \sigma,
    \end{equation}
    where 
    \begin{equation} 
        \mathrm{w}( \mathrm{\sigma}) =   \prod_i \frac{\pi}{ \cosh (\pi  \sigma_i )},
    \end{equation}
and  $R(\sigma)$ is a polynomial on $\sigma = (\sigma_1, \cdots, \sigma_n)$.


From \cite[Section 3.1]{BenWinter}, we see that an analogous twisted trace can be defined for quantized Coulomb branch algebras \( \mathcal{A}_\hbar(T, \mathbf{N}) \), where \( T = (\mathbb{C}^*)^m \) and \( \mathbf{N} = \mathbb{C}^n \) with \( n > m \), under the fixed embedding \( T^m \hookrightarrow T^n \) (as specified previously).  
This twisted trace is constructed via an integral over a coset of the real Lie algebra \( \mathfrak{t}_\mathbb{R}^m \).  
In coordinates, this means for \ref{prop:TnCn}, restricting to the first $m$ coordinates of $(\sigma_1,\cdots,\sigma_m,\ms)$.
\begin{equation}
    \Tr(R(z))=\int_{\mathbb{R}^n} e^{2 \pi i \zeta(\sigma)} R(\sigma) \mathrm{w}(\sigma)\mathrm{d} \sigma.
\end{equation}
We choose a nontrivial fixed $\ms \in \ff $. Then, as in Section \ref{sec:abelian}, we write elements in $\ft ^ m$ as elements of a coset in $ \ft ^n$. Then, the twisted trace \eqref{lemma:coulomb_representation_ndim} is
    \begin{equation} \label{prop:coulomb_trace_invariant_nontrivial_flavor}
        \Tr(R(z))=\int_{\mathbb{R}^m} e^{2 \pi \zeta (\sigma)} R(\sigma) \mathrm{w}(\sigma+\m)\mathrm{d} \sigma,
    \end{equation}
    where $R(\sigma)$ is a polynomial in the coordinates $\sigma = (\sigma_1, \cdots, \sigma_n)$, and
    \begin{equation} \label{eq:w}
      \mathrm{w}(\sigma+\m) =  \prod_w \prod_i \frac{\pi}{ \cosh (\pi w(\sigma_i+\m))}.
    \end{equation}



We recall $z =\frac{1}{2} - i \si + \frac{b}{2}$. Positive roots of $\GL_n$ are in the form of $\beta = z_i-z_j$, for $i>j$.
\begin{Lemma} \label{rmk}
For conical theories, all images of monopole operators under $\alpha_C$ and $\bar{\alpha}_C$ acting on the ground state $\IGN \in \FP ^n$ lie in $L^2$ under the inner product \eqref{eq:GNprd}. 
\end{Lemma}

\begin{proof}
The norm of the $\Gamma$ function decays exponentially in the imaginary direction.
The conical condition defined in \Cref{1.1} ensures the exponential decay of the ground state \cref{eq:1withbeta} as \(\|\sigma\| \to \infty\) (see \cite[Section 5.1]{GaiottoTempFile} for the reasoning).

The action of \(\alpha_C(\mathbf{V}_\lambda)\) on \(\IGN\) yields a state \(\psi_\lambda = \alpha_C(\mathbf{V}_\lambda) \IGN\) expressible as a finite sum over \(W/W_\lambda\):
\[
\psi_\lambda(\sigma, b) = \sum_{w \in W / W_\lambda} \psi_{\lambda,w}(\sigma, b), \quad \psi_{\lambda,w}(\sigma, b) = \frac{Q_w(\sigma) \cdot \widetilde{\psi}_w(\sigma, b)}{\prod_{\substack{\beta \in \Delta^+ \\ \langle \lambda, \beta \rangle \neq 0}} w\beta(z)},
\]
where \(Q_w(\sigma)\) is a polynomial in \(\sigma\), and \(\widetilde{\psi}_w\) is the ground state evaluated at shifted arguments \((\sigma', b')\) determined by \(\lambda\) and \(w\).

The \(L^2\)-norm of \(\psi_\lambda\) is:
\[
\|\psi_\lambda\|^2 = \sum_{b \in \mathbb{Z}^m} \int_{\mathbb{R}^m} |\psi_\lambda(\sigma, b)|^2  d\sigma = \sum_{b \in S} \int_{\mathbb{R}^m} |\psi_\lambda(\sigma, b)|^2  d\sigma,
\]
where \(S = \mathrm{supp}(\psi_\lambda)\) is finite. For each \(b \in S\), the integral is over a finite sum of terms \(\psi_{\lambda,w}\). We use $[G]$ to denote products of gamma functions. Each \(\psi_{\lambda,w}\) has the form:
\[
|\psi_{\lambda,w}(\sigma, b)| = \left| \frac{Q_w(\sigma) \cdot 
[G]}{ \prod \Gamma(i\beta)\Gamma(-i\beta) \cdot \prod w\beta} \right|.
\]
By the conical condition, the numerator decays exponentially as \(\|\sigma\| \to \infty\) due to the gamma functions \(\Gamma\left(\frac{1}{2} - i w(\sigma + \mathbf{m})\right)\).
 The denominator grows polynomially since  \(\prod w\beta\) are products of linear functions, and  \(\Gamma(i\beta)\Gamma(-i\beta)\) grows polynomially by Stirling’s approximation.
The function \(Q_w(\sigma)\) is polynomial in \(\sigma\).

Thus, there exist constants \(C_w, A_w > 0\) such that:
\[
|\psi_{\lambda,w}(\sigma, b)| \leq C_w e^{-A_w \|\sigma\|}, \quad \forall \sigma \in \mathbb{R}^m.
\]
Squaring and integrating gives:
\[
\int_{\mathbb{R}^m} |\psi_{\lambda,w}(\sigma,b)|^2  d\sigma \leq C_w^2 \int_{\mathbb{R}^m} e^{-2A_w \|\sigma\|}  d\sigma < \infty.
\]
Summing over the finite sets \(S\) and \(W/W_\lambda\) confirms \(\|\psi_\lambda\|^2 < \infty\). The same argument holds for \(\bar{\alpha}_C(\mathbf{V}_\lambda) |1\rangle_{G,\mathbf{N}}\) by symmetry.  
\end{proof}

We denote the adjoint of $|v\rangle $ by $ \bign $ such that for the pairing given by Haar measure \eqref{eq:pairing} on $\FP^{m}$,
\begin{equation} 
\bign = \langle |v \rangle , \bullet \rangle.
\end{equation}

\begin{Lemma} \label{lemma:sphere4.2}
Let $\lambda$ be a coweight, the following conditions hold for $w_0\in W$, the longest element.
  \begin{align}
\alpha_C\left(\bV_\lambda\right)\IGN & =\bar{\alpha}_C\left(\bV_{-w_0\lambda}\right)\IGN, & \alpha_C\left(\bV_{-w_0 \lambda}\right)\IGN & =\bar{\alpha}_C\left(\bV_{\lambda}\right)\IGN,\label{eq:IGN} \\
\bign \alpha_C\left(\bV_\lambda\right) & =-\bign \bar{\alpha}_C\left(\bV_{-w_0\lambda}\right), & \bign \alpha_C\left(\bV_{-w_0\lambda}\right) & =-\bign \bar{\alpha}_C\left(\bV_{\lambda}\right). \label{eq:IGNconju}
\end{align}  
\end{Lemma}
\begin{proof}
   Later in the proof, we will use the equations of meromorphic functions
\begin{equation} \label{eq:sinh}
   \frac{1}{\prod_{i \in Q_0} \prod_{\beta \in \Delta_i} \beta} \prod_{i \in Q_0} \prod_{\beta \in \Delta_{i,+}} \frac{1}{\Gamma(i\beta)\Gamma(-i\beta) } = \prod_{i \in Q_0} \prod_{\beta \in \Delta_{i,+}} \frac{\sinh(\pi \beta)}{\pi \beta}.
\end{equation}

 For a general abelian theory, this is \Cref{lm:alphacwm}.

 The action of nonabelian monopole operators can be written in terms of the abelian monopole operators using the formula 
 \cref{eq:opg}. 
 Let $e$ be a fundamental coweight. For better illustration, we provide a proof for $\lambda = e$. The analogous proof for a general $\lambda$ holds.
 The following holds for all $\psi \in \GP$.
\begin{equation}
\label{eq:alpha_C_expression} 
\begin{aligned}
\alpha_C\left(\mathbf{V}_{-w_0e}\right)\psi   
&= \sum_{w \in W/W_{-w_0e}} \frac{\prod_r P_{i,r}^{\xi_{i,r}\left(w ({-w_0e})\right)}}{\displaystyle \prod_{\substack{\beta \in \Delta^+ \\ \langle {-w_0e}, \beta \rangle \neq 0}} w\beta} \psi \\
&= \sum_{(ww_0) \in W/W_{-w_0 e}} \frac{\prod_r P_{i,r}^{\xi_{i,r}\left(w w_0 (-w_0 e)\right)}}{\displaystyle \prod_{\substack{\beta \in \Delta^+ \\ \langle -w_0 e, \beta \rangle \neq 0}} (ww_0)\beta} \psi \\
&= \sum_{w \in W/W_{e}} \frac{\prod_r P_{i,r}^{\xi_{i,r}\left(-w e)\right)}}{\displaystyle \prod_{\substack{\beta \in \Delta^+ \\ \langle -w_0 e, \beta \rangle \neq 0}} w\beta} \psi.
\end{aligned}
\end{equation}
For $\IGN$, by \cref{eq:sinh}, and \Cref{lm:alphacwm}, the equation \eqref{eq:alpha_C_expression}  gives
\begin{equation}
 \alpha_C\left(\mathbf{V}_{-w_0e}\right)\IGN    = \sum_{w \in W/W_{e}} \frac{\prod_r -\widetilde{P}_{i,r}^{\xi_{i,r}\left(-w e)\right)}}{\displaystyle \prod_{\substack{\beta \in \Delta^+ \\ \langle -w_0 e, \beta \rangle \neq 0}} w\beta} \IGN =\bar{\alpha}_C\left(\bV_e\right) \IGN .
\end{equation}

By a similar computation,
\begin{equation}  \alpha_C\left(\bV_e\right)\IGN =\bar{\alpha}_C\left(\bV_{-w_0e}\right)\IGN .
\end{equation} 
holds.
 
Similarly, the adjoint relations \eqref{eq:IGNconju} follow from \Cref{lm:alphacwm}  and the \(W\)-invariance of $\bign$ .
\end{proof}


Now we are ready to prove the main theorem.
\begin{proof}[Proof of \Cref{prop:pureg-trace}]
We give the proof for $G= \GL_n$, and for multiple copies, an analogous proof holds. 
Assume  $(G, \mathbf{N})$  is conical. This ensures convergence of the integral \cref{eq:trthm} via exponential decay (see  \Cref{rmk}).
To show \cref{eq:GNTrw} holds, we first recall \cref{eq:opg}. Since $\cA_\hbar(G, \bN)$ is generated by $\bV_\lambda$'s, it is enough to consider $\alpha_C (\bV_\lambda )$ for \cref{eq:GNTrw}.

By \Cref{lemma:sphere4.2}, and an analogous argument of \cite[Lemma 2.3]{BenWinter}, the equation \cref{eq:GNTrw} holds.
By \cref{eq:sinh},
\begin{equation}
\begin{split}
   \bign g\alpha_C(R (\sigma)) \IGN
    &
     = \int_{\RR^m} e^{2 \pi i \zeta(\sigma)} R(\sigma)  \frac{\prod_w \Gamma\left(\frac{1}{2}+i w(\sigma+\m)\right)}{\prod_{\beta \in \Delta_{+}} \Gamma(i \beta) \Gamma(-i \beta) \beta} \cdot
    \frac{\prod_w \Gamma\left(\frac{1}{2}-i w(\sigma+\m)\right)}{\prod_{\beta \in \Delta_{+}} \Gamma(i \beta) \Gamma(-i \beta) \beta} d \sigma  \\ &
       =\int_{\RR^m} e^{2 \pi i \zeta(\sigma)} R(\sigma) \frac{\prod_{\beta \in \Delta_{+}} (\frac{1}{\pi} \sinh (\pi \beta))^2}{\prod_{w } \frac{1}{\pi} \cosh (\pi w(\sigma+\m))} d \sigma .
\end{split}
\end{equation}
Substituting \cref{eq:w} gives \cref{eq:trthm} from \cref{eq:GNTrw}.

The uniqueness of the twisted trace is given by the uniqueness of the twisted trace of abelian theories and \cref{368}.
\end{proof}

Re-writing this proposition in the language of quiver gauge theory gives us \Cref{prop:vector}.
\begin{Remark}
    The operator representation exists for all type A quantized Coulomb branches; however, this trace only exists for quantized Coulomb branches of conical theories when \cref{eq:trthm} converges.
\end{Remark}
\section{Example}
\subsection{Example: $G =\CC^*, \bN = \CC^2$ }
\label{sec:C*N2}
For representation of $\mathcal{A}_\hbar (G, \bN) $, we consider the following exact sequence on the Lie algebras:
\begin{equation}
    1 \rightarrow \CC \rightarrow \CC^2 \rightarrow \CC \rightarrow 1.
  \end{equation}
  We recall the weights are $\pm 1$.
By \ref{lemma:coulomb_representation_nontrivial_flavor}, we have 
\begin{equation}
    |1\rangle_{\CC, \CC^2}=\delta_{b, 0} \Gamma\left(\frac{1}{2}-i\sigma+\m)\right)
    \Gamma\left(\frac{1}{2}+i(\sigma+\m)\right)  ,
    \end{equation}
for coweights $\beta_1$ and $\beta_2$.
The integral form of this twisted trace given by \eqref{prop:coulomb_trace_invariant_nontrivial_flavor} is 
 \begin{equation} 
        \Tr (R(z))=\int_{\mathbb{R}} e^{2 \pi \zeta (\sigma)} R(\sigma) \mathrm{w}(\sigma+\m)\mathrm{d} \sigma,
    \end{equation}
with 
    \begin{equation}
      \mathrm{w}(\sigma+\m) =  \frac{\pi^2}{ \cosh (\pi 
      (\sigma+\m)) \cosh (\pi 
     (-1) (\sigma+\m))}.
    \end{equation}

\subsection{Example: $G= \GL_2$, $\bN= \CC^2$}
The maximal torus of $\GL_2$ is $T=(\CC^*)^2$. We denote the coweights as $\sigma_1$ and $\sigma_2$. They are coordinates of $\sigma$, viewed as functions over $\ft$. Therefore $\sigma_i$'s are in the dual of $\ft$. The roots of $\GL_2(\mathbb{C})$ are $\pm\left(z_1-z_2\right) = \pm\left( \frac{1}{2}(b_1 - b_2) - i(\si_1-\si_2)\right) $, forming a root system of type $A_1$.

We denote $\C = \left( \frac{1}{2}(b_1 - b_2) - i(\si_1-\si_2)\right)$, $\D=\sigma_1-\sigma_2$.

By Proposition \ref{prop:1gauge},
\begin{equation}
|v'\rangle_{\GL_2, \CC^2}=\delta_{b_1, 0} \delta_{b_2,0}\frac{\Gamma\left(z_1\right)\Gamma\left(z_2\right)}{ \Gamma( \C)^2 }.
\end{equation}

The twisted trace for $\cA_\hbar (\GL_2 , \CC^2 )$ is given by
\begin{equation}  
\Tr(R(\sigma)) = \int_{\RR^2} e^{2\pi \zeta(\sigma)} R(\sigma) \, \mathrm{w}(\sigma + \m) \, d\sigma,  
\end{equation}  
where $R(\sigma)\in P_{\GL_2, \CC^2}$,
\begin{equation}
\mathrm{w}(\sigma + \m) =H =\delta_{b, 0} \frac{\Gamma\left(\frac{1}{2}+i \sigma_1\right) \Gamma\left(\frac{1}{2}+i \sigma_2\right) \Gamma\left(\frac{1}{2}-i \sigma_1\right) \Gamma\left(\frac{1}{2}-i \sigma_2\right) }{\Gamma\left(-i\D)\right)^2 \Gamma\left(i\D\right)^2 \D ^2} =  
\frac{\sinh ^2\left(\pi\D \right)}{\cosh \left(\pi \sigma_1\right) \cosh \left(\pi \sigma_2\right)}.
\end{equation}

We want to compute:
\begin{equation}
\alpha_C\left(\mathbf{V}_{(0,-1)}\right) \alpha_C\left(\mathbf{V}_{(1,0)}\right)\IGN.
\end{equation}

For $G = \GL_2$, $\la = (1, 0)$, the Weyl group is $\Weyl = S_2$ and the stabilizer $\Weyl_\la$ is trivial. The only positive root is $\beta = e_1 - e_2$, with $\langle \la, \beta \rangle = 1 \neq 0$. Thus, the formula simplifies to:
\[
[\alpha_C(\bV_{(1,0)}) \psi](\bs, \mathbf{b}) = \frac{[\alpha_C(r^{(1,0)}) \psi](\bs, \mathbf{b}) - [\alpha_C(r^{(0,1)}) \psi](\bs, \mathbf{b})}{\C},
\]
where $z_i = \tfrac{1}{2} - i\si_i + \tfrac{1}{2}b_i$ are the complex weight variables.

By \cref{798}, we have:

$$
\begin{aligned}
& \alpha_C\left(r^{1,0}\right)\IGN=X_1\IGN=\delta_{0, b_2} \delta_{1, b_1} \frac{\Gamma\left(\frac{1}{2}\left(1+b_1-2 i \sigma_1\right)\right) \Gamma\left(\frac{1}{2}-i \sigma_2\right)}{\Gamma\left(\left(z_1 - z_2\right)\right)^2 }, \\
& \alpha_C\left(r^{0,1}\right)\IGN=X_2\IGN=\delta_{0, b_1} \delta_{1, b_2} \frac{\Gamma\left(\frac{1}{2}\left(1+b_2-2 i \sigma_2\right)\right) \Gamma\left(\frac{1}{2}-i \sigma_1\right)}{\Gamma\left(\left(z_1 - z_2\right)\right)^2 }.
\end{aligned}
$$
Then,
$$
\alpha_C\left(\mathbf{V}_{(1,0)}\right)|v'\rangle=\frac{1}{z_1-z_2}\left[\delta_{0, b_2} \delta_{1, b_1} \frac{\Gamma\left(1-i \sigma_1\right) \Gamma\left(\frac{1}{2}-i \sigma_2\right)}{\Gamma\left(i\left(z_1-z_2\right)\right)^2}-\delta_{0, b_1} \delta_{1, b_2} \frac{\Gamma\left(1-i \sigma_2\right) \Gamma\left(\frac{1}{2}-i \sigma_1\right)}{\Gamma\left(i\left(z_1-z_2\right)\right)^2}\right] .
$$
The operators $P_i$ act as:
$$
\left(P_i \psi\right)(\sigma, b)=\left(\frac{1}{2}+i \sigma_i+\frac{b_i}{2}\right) \psi\left(\sigma_i-\frac{i}{2}, b_i+1\right).
$$

We know that 

$$
\alpha_C\left(\mathbf{V}_{(0,-1)}\right) \alpha_C\left(\mathbf{V}_{(1,0)}\right)\IGN=\left(\frac{P_2-P_1}{z_1-z_2}\right)\left(\alpha_C\left(\mathbf{V}_{(1,0)}\right)\IGN \right).
$$

Let:

$$
\Psi=\alpha_C\left(\mathbf{V}_{(1,0)}\right)\IGN= \delta_{b_2, 0} \delta_{b_1, 1} A(\si) +\delta_{b_1, 0} \delta_{b_2, 1} B (\si),
$$

where

$$
\begin{aligned}
& A(\sigma)=\frac{\Gamma\left(1-i \sigma_1\right) \Gamma\left(\frac{1}{2}-i \sigma_2\right)}{\C \Gamma\left(\C \right)^2}, \\
& B(\sigma)=-\frac{\Gamma\left(1-i \sigma_2\right) \Gamma\left(\frac{1}{2}-i \sigma_1\right)}{\C \Gamma\left(\C\right)^2}.
\end{aligned}
$$
Then:
$$
\alpha_C\left(\mathbf{V}_{(0,-1)}\right) \Psi=\frac{1}{z_1-z_2}\left(P_2 \Psi -P_1\Psi\right).
$$
Let $\Delta=\sigma_1-\sigma_2$. Then, for $b = (0,0)$, $$ z_1 - z_2 = i\D.$$
For $\left(b_1, b_2\right)=(1,-1)$, $$ z_1 - z_2 = 1- i \D.$$
For $\left(b_1, b_2\right)=(-1,1)$, $$ z_1 - z_2 = -1- i \D .$$
For $(b_1,b_2)= (0,0)$, the term is:
$$
-\frac{ i \,\big(-i + \sigma_{1} + \sigma_{2}\big)\, \Gamma\!\left(\tfrac{1}{2} - i\sigma_{1}\right)\, \Gamma\!\left(\tfrac{1}{2} - i\sigma_{2}\right)}{\left[\Gamma\!\left(1 - i\sigma_{1} + i\sigma_{2}\right)\right]^{2}}.
$$
For $\left(b_1, b_2\right)=(1,-1)$, the term is:
\[
-\frac{\Gamma\!\left(1 - i\sigma_{1}\right) \, \Gamma\!\left(1 - i\sigma_{2}\right)}{\left[\Gamma\!\left(2 - i\sigma_{1} + i\sigma_{2}\right)\right]^{2}}.
\]
For $\left(b_1, b_2\right)=(-1,1)$, the term is:
\[
-\frac{\Gamma\!\left(1 - i\sigma_{1}\right)\,\Gamma\!\left(1 - i\sigma_{2}\right)}{\left[\Gamma\!\left(-i(\sigma_{1}-\sigma_{2})\right)\right]^{2}}.
\]

After simplification, we have:
\begin{equation}
\begin{aligned}
\alpha_C\left(\mathbf{V}_{(0,-1)}\right) \Psi = - \delta_{b_1,0} \delta_{b_2,0}
\frac{ i \,\big(-i + \sigma_{1} + \sigma_{2}\big)\, \Gamma\!\left(\tfrac{1}{2} - i\sigma_{1}\right)\, \Gamma\!\left(\tfrac{1}{2} - i\sigma_{2}\right)}{\left[\Gamma\!\left(1 - i\sigma_{1} + i\sigma_{2}\right)\right]^{2}}  \\
-
 \delta_{b_1,1} \delta_{b_2,-1}\frac{\Gamma\!\left(1 - i\sigma_{1}\right) \, \Gamma\!\left(1 - i\sigma_{2}\right)}{\left[\Gamma\!\left(2 - i\sigma_{1} + i\sigma_{2}\right)\right]^{2}}\\
- \delta_{b_1,-1} \delta_{b_2,1}\frac{\Gamma\!\left(1 - i\sigma_{1}\right)\,\Gamma\!\left(1 - i\sigma_{2}\right)}{\left[\Gamma\!\left(-i(\sigma_{1}-\sigma_{2})\right)\right]^{2}} .
\end{aligned}
\end{equation}

The twisted trace is
\begin{equation}
\operatorname{Tr}\left(\mathbf{V}_{(0,-1)} \mathbf{V}_{(1,0)}\right)=\langle v| g \alpha_C\left(\mathbf{V}_{(0,-1)} \mathbf{V}_{(1,0)}\right)|v'\rangle .
\end{equation}
Since $|1\rangle$ is supported only at $b=(0,0)$, this becomes
$$
\operatorname{Tr}=\int_{\mathbb{R}^2} \langle v| \cdot e^{2 \pi i \zeta(\sigma)} \cdot\left[\alpha_C\left(\mathbf{V}_{(0,-1)} \mathbf{V}_{(1,0)}\right)|v'\rangle\right]_{(b=(0,0))} d \sigma_1 d \sigma_2.
$$
Substituting the expressions, we obtain
$$\begin{aligned}  |v\rangle(\sigma, 0)& =\frac{\Gamma\left(\frac{1}{2}-i \sigma_1\right) \Gamma\left(\frac{1}{2}-i \sigma_2\right)}{\Gamma(-i \Delta)^2(-i \Delta)^2}, \\  \left[\alpha_C\left(\mathbf{V}_{(0,-1)} \mathbf{V}_{(1,0)}\right)|v'\rangle\right]_{(0,0)} &
= - \frac{ i \,\big(-i + \sigma_{1} + \sigma_{2}\big)\, \Gamma\!\left(\tfrac{1}{2} - i\sigma_{1}\right)\, \Gamma\!\left(\tfrac{1}{2} - i\sigma_{2}\right)}{\left[\Gamma\!\left(1 - i\sigma_{1} + i\sigma_{2}\right)\right]^{2}}  \\ &
=   \frac{ - \,\big(1 + i \sigma_{1} + i \sigma_{2}\big)\, \Gamma\!\left(\tfrac{1}{2} - i\sigma_{1}\right)\, \Gamma\!\left(\tfrac{1}{2} - i\sigma_{2}\right)}{\left[\Gamma\!\left(1 - i\sigma_{1} + i\sigma_{2}\right)\right]^{2}} .\end{aligned}$$
we obtain:
\begin{equation}\label{5.11}
    \begin{aligned}
     &\operatorname{Tr} \left(\mathbf{V}_{(0,-1)} \mathbf{V}_{(1,0)}\right) = \\
    & =\int_{\mathbb{R}^2} e^{2 \pi i \zeta(\sigma)} \cdot \frac{\left(-1-i\left(\sigma_1+\sigma_2\right)\right)}{\Gamma(-i \Delta)^2(-i \Delta)^2 \Gamma^2(1-i \Delta)} \cdot\left|\Gamma\left(\frac{1}{2}-i \sigma_1\right)\right|^2\left|\Gamma\left(\frac{1}{2}-i \sigma_2\right)\right|^2 d \sigma_1 d \sigma_2 \\  & =
     \int_{\mathbb{R}^2} e^{2 \pi i \zeta(\sigma)} \cdot \frac{\left(-1-i\left(\sigma_1+\sigma_2\right)\right)}{ \Gamma^4(1-i \Delta)} \cdot\left|\Gamma\left(\frac{1}{2}-i \sigma_1\right)\right|^2\left|\Gamma\left(\frac{1}{2}-i \sigma_2\right)\right|^2 d \sigma_1 d \sigma_2 .
    \end{aligned} 
\end{equation}
The second line of \cref{5.11} is obtained from the first line by \cref{funcgamma}. This expression shows that the integrand of \cref{5.11} is entire. 

Using the identities \cref{funcgamma} and
$$
\begin{aligned}
& \left|\Gamma\left(\frac{1}{2}-i \sigma\right)\right|^2=\frac{\pi}{\cosh (\pi \sigma)} \text {, } \\
& |\Gamma(1+ i \D)|^2=\frac{\pi \D}{\sinh \pi \D }\text {, }
\end{aligned}
$$
we compute the denominator:
$$
\Gamma(-i \Delta)^2(-i \Delta)^2 \Gamma(1-i \Delta)^2= \frac{\pi^2 \D^2 }{ \sinh ^2(-\pi \D)} = \frac{\pi^2 \D^2 }{ \sinh ^2(\pi \D)}.
$$
Also, the numerator includes:
$$
\left|\Gamma\left(\frac{1}{2}-i \sigma_1\right)\right|^2\left|\Gamma\left(\frac{1}{2}-i \sigma_2\right)\right|^2=\frac{\pi^2}{\cosh \left(\pi \sigma_1\right) \cosh \left(\pi \sigma_2\right)} .
$$
Therefore, the integrand simplifies to:
$$
e^{2 \pi i \zeta(\sigma)} \cdot\left(-1-i \left(\sigma_1+\sigma_2\right)\right) \cdot \frac{1}{\Delta^2} \cdot \frac{\sinh ^2(\pi \Delta)}{\cosh \left(\pi \sigma_1\right) \cosh \left(\pi \sigma_2\right)} .
$$
The twisted trace is given by:
\begin{equation} 
\operatorname{Tr}\left(\mathbf{V}_{(0,-1)} \mathbf{V}_{(1,0)}\right)=\int_{\mathbb{R}^2} e^{2 \pi i \zeta(\sigma)}\left(-1-i\left(\sigma_1+\sigma_2\right)\right) \cdot \frac{1}{(\sigma_1-\sigma_2)^2} \cdot \frac{\sinh ^2\left(\pi\left(\sigma_1-\sigma_2\right)\right)}{\cosh \left(\pi \sigma_1\right) \cosh \left(\pi \sigma_2\right)} d \sigma_1 d \sigma_2 .
\end{equation}

\bibliographystyle{plainnat}
\bibliofont
\bibliography{analyticsymplecticdualityref}

\end{document}